\documentclass[a4paper,twoside]{amsart}

\usepackage{amssymb, tabu, url}
\usepackage{ifpdf}
\ifpdf
  \usepackage[colorlinks,linkcolor=red,anchorcolor=blue,citecolor=green]{hyperref}
\else
\fi

\ifx\isdraft\undefined
\else
    \usepackage{fancyhdr}
\fi

\numberwithin{equation}{section}

\newtheorem{theorem}{Theorem}[section]

\newtheorem{claim}[theorem]{Claim}

\newtheorem{conjecture}[theorem]{Conjecture}
\newtheorem{corollary}[theorem]{Corollary}

\newtheorem{lemma}[theorem]{Lemma}

\newtheorem{proposition}[theorem]{Proposition}

\theoremstyle{definition}
\newtheorem{definition}[theorem]{Definition}

\newcommand{\<}{\langle}
\newcommand{\uh}{\upharpoonright}
\renewcommand{\>}{\rangle}

\newcommand{\dom}{\operatorname{dom}}

\newcommand{\RCA}{\operatorname{RCA}_0}
\newcommand{\ACA}{\operatorname{ACA}_0}
\newcommand{\WKL}{\operatorname{WKL}_0}
\newcommand{\ATR}{\operatorname{ATR}_0}
\newcommand{\PiCA}{\Pi^1_1\operatorname{-CA}}

\newcommand{\RT}{\operatorname{RT}}
\newcommand{\COH}{\operatorname{COH}}

\newcommand{\EM}{\operatorname{EM}}
\newcommand{\SEM}{\operatorname{SEM}}

\newcommand{\ADS}{\operatorname{ADS}}

\newcommand{\SADS}{\operatorname{SADS}}

\newcommand{\ART}{\operatorname{ART}}
\newcommand{\FS}{\operatorname{FS}}
\newcommand{\TS}{\operatorname{TS}}
\newcommand{\STS}{\operatorname{STS}}

\newcommand{\RRT}{\operatorname{RRT}}

\begin{document}

\title{The Definability Strength of Combinatorial Principles}


\author{Wei Wang}

\thanks{This research is partially supported by an NCET grant from MOE of China and National Fund of Philosophy and Social Science of China (Project 13\&ZD186). The author thanks Ludovic Patey for inspiring conversations. Parts of this paper were presented on the workshop Computability Theory and Foundation of Mathematics workshop in February 2014 in Tokyo, where the author received many helpful comments and interesting questions. The author thanks the organizers for their generous support.}

\address{Institute of Logic and Cognition and Department of Philosophy, Sun Yat-sen University, 135 Xingang Xi Road, Guangzhou 510275, P.R. China}
\email{wwang.cn@gmail.com}

\begin{abstract}
We introduce the definability strength of combinatorial principles. In terms of definability strength, a combinatorial principle is strong if solving a corresponding combinatorial problem could help in simplifying the definition of a definable set. We prove that some consequences of Ramsey's Theorem for colorings of pairs could help in simplifying the definitions of some $\Delta^0_2$ sets, while some others could not. We also investigate some consequences of Ramsey's Theorem for colorings of longer tuples. These results of definability strength have some interesting consequences in reverse mathematics, including strengthening of known theorems in a more uniform way and also new theorems.
\end{abstract}

\ifx\isdraft\undefined
\else
    \today
\fi

\maketitle

\section{Introduction}\label{s:Introduction}

In early ages of reverse mathematics, people found that many classical theorems in ordinary mathematics, when formulated in second order arithmetic, are equivalent to certain subsystems of second order arithmetic over the Recursive Comprehension Axiom ($\RCA$), in terms of their \emph{provability strength}. The most prominent subsystems are the so-called \emph{big five}: $\RCA$, $\WKL$, $\ACA$, $\ATR$ and $\PiCA$, with their provability strength growing strictly stronger from left to right. So the big five give us a nice ruler, against which the provability strength of many classical theorems can be precisely measured. But there are exceptions. One of these exceptions is the instance of Ramsey's Theorem for $2$-colorings of pairs, denoted by $\RT^2_2$. From Jockusch \cite{Jockusch:1972.Ramsey}, we can see that every instance of Ramsey's Theorem is a consequence of $\ACA$, the instance for colorings of triples is equivalent to $\ACA$ over $\RCA$, but $\RT^2_2$ is not implied by $\WKL$; later Seetapun \cite{Seetapun.Slaman:1995.Ramsey} proved that $\RT^2_2$ is strictly weaker than $\ACA$ over $\RCA$. Since Seetapun's work, people have found many propositions in second order arithmetic related more or less to $\RT^2_2$ whose provability strength cannot be precisely measured by the ruler. People started comparing the provability strength of these propositions to each other and have revealed a very complicated picture. For a general impression of this complicated picture, we refer the reader to the Reverse Mathematics Zoo (\url{http://rmzoo.uconn.edu/}) maintained by Dzhafarov.

Most propositions in the complicated picture are combinatorial principles and can be formulated as $\Pi^1_2$ sentences, i.e., sentences of the form $\Phi = (\forall X)(\exists Y)\varphi(X,Y)$ where $\varphi$ is arithmetic. Given such $\Phi$, each $X$ represents an \emph{instance} of the corresponding combinatorial problem, and each $Y$ satisfying $\varphi(X,Y)$ a \emph{solution}. A popular and fruitful approach to examine the provability strength of $\Phi$ is by analyzing its \emph{computability strength}. If for each set $W$ in a certain class of non-computable sets, there exists a computable $\Phi$-instance $X$ such that every solution $Y$ to the instance $X$ can code $W$ in some effective way, then we may say that $\Phi$ has strong computability strength in some sense; otherwise, $\Phi$ is considered weak. By relativization, usually we can build an $\omega$-model of a base system (e.g., $\RCA$) and a $\Pi^1_2$ proposition $\Psi$ with weak computability strength, which does not contain any solution to a computable instance of another $\Pi^1_2$ proposition $\Phi$ with strong computability strength. So we conclude that $\Psi$ does not imply $\Phi$ over the base system. The analysis of computability strength is not limited to comparing propositions with different provability strength, but it can also help us comparing propositions that have equal provability strength, as shown in \cite{Dorais.Dzhafarov.ea:ta}.

In this paper, we introduce a new kind of analysis, based on what we call the \emph{definability strength}. Roughly, if $\Phi$ is a $\Pi^1_2$ sentence then the definability strength of $\Phi$ is measured by whether solving a $\Phi$-instance helps in simplifying certain definability problem. A formal definition is given in Definition \ref{def:presv-arith}. Here we mainly apply this analysis to $\Pi^1_2$ propositions in Ramsey theory and also related propositions studied in reverse mathematics.

As the analysis of computability strength, analyzing the definability strength of $\Pi^1_2$ propositions also leads to consequences in reverse mathematics. We shall present several results of this kind here. These new results introduce more chaos to the Reverse Mathematics Zoo. However, they also give us a rather clear classification of most animals in the Zoo by definability strength. Interestingly, the analysis of definability strength also yields new proofs of known reverse mathematics theorems, which were obtained by the analysis of computability strength. Though this new analysis sounds a little coaser than that of computability strength, it gives some new proofs in a more uniform way. For example, people have proved several theorems concerning the provability strength of the Ascending or Descending Sequence principle ($\ADS$), and some of these proofs share little similarity. But through the analysis of definability strength we obtain new proofs which all depend on the definability strength of $\ADS$. Moreover, as definability appears naturally in various areas of logic and the new analysis connects combinatorial principles to definability problems, we believe that this new analysis is interesting in its own right.

Below, we briefly introduce the remaining parts of this paper:
\begin{itemize}
  \item In \S 2, we give a formal definition of the center concept of this paper and prove some general facts which will facilitate our concrete analysis.
  \item In \S 3, we study some computability notions and $\Pi^1_2$ propositions which are weak in terms of definability strength.
  \item In \S 4, we show that some other $\Pi^1_2$ propositions are strong in terms of definability strength.
  \item In \S 5, we conclude this paper with a summarization of the definability strength results, some consequences in reverse mathematics and a few remarks.
\end{itemize}

We finish this section with a few words on notation and background knowledge.

If $s$ and $t$ are two finite sequences, then we write $st$ for the concatenation of $s$ and $t$. If $x$ is a single symbol, then $\<x\>$ is the finite sequence with only one symbol $x$. The length of a finite sequence $s$ is denoted by $|s|$. If $l < |s|$ then $s \uh l$ is the initial segment of $s$ of length $l$. For $X \subseteq \omega$, $X \uh l$ is interpreted as an initial segment of the characteristic function of $X$ in the obvious way.

Recall that $[X]^r$ for $0 < r < \omega$ is the set of $r$-element subsets of $X$. We also write $[X]^\omega$ for the set of countable subsets of $X$; $[X]^{< r}, [X]^{\leq r}, [X]^{< \omega}, [X]^{\leq \omega}$ are interpreted naturally. If $X \subseteq \omega$, then elements of $[X]^{\leq \omega}$ are identified with strictly increasing sequences. We use $\sigma, \tau, \ldots$ for elements of $\omega^{< \omega}$. Under the above convention, we may perform both sequence operations and set operations on elements of $[\omega]^{<\omega}$. For example, we can write $\sigma\tau$ for $\sigma \cup \tau$, if $\sigma$ and $\tau$ are in $[\omega]^{<\omega}$ and $\max \sigma < \min \tau$; $\sigma \subseteq \tau$ if $\sigma$ is a subset of $\tau$; and $\sigma - \tau = \{x \in \sigma: x \not\in \tau\}$. We extend this convention to infinite subsets of $\omega$, so we write $\sigma X$ for $\sigma \cup X$, if $\max \sigma < \min X$ and $X \in [\omega]^{\leq \omega}$. We fix a computable bijection $\ulcorner \cdot \urcorner: \omega^{<\omega} \to \omega$ and occasionally identify $\sigma$ with $\ulcorner \sigma \urcorner$. So we may write $\sigma < \tau$ for $\ulcorner \sigma \urcorner < \ulcorner \tau \urcorner$, etc.

For two sets $X$ and $Y$, we write $X \subseteq^* Y$ if $X - Y$ is finite and $X =^* Y$ if $X \subseteq^* Y$ and $Y \subseteq^* X$.

When we work on Ramsey theory, we call a function as a \emph{coloring}. For a positive integer $c$, a \emph{$c$-coloring} is a coloring with range contained in $c = \{0, 1, \ldots, c-1\}$. A \emph{homogeneous} set of a coloring $f$ on $[\omega]^n$ is a set $H$ such that $f$ is constant on $[H]^n$. \emph{Ramsey's Theorem} states that for every positive integers $c$ and $n > 1$ every $c$-coloring of $[\omega]^n$ admits an infinite homogeneous set. $\RT^n_c$ stands for the instance of Ramsey's Theorem for fixed $n$ and $c$. Sometimes it is helpful to consider \emph{stable} colorings: a coloring $f: [\omega]^{n+1} \to \omega$ is \emph{stable} if $\lim_x f(\sigma\<x\>)$ exists for all $\sigma \in [\omega]^n$.

It is widely understood that \emph{computable} and \emph{recursive} are synonymous and so are \emph{computability} and \emph{recurion theory}. Here we prefer \emph{computable} and \emph{computability} in most cases, since \emph{computability strength} aligns better with its provability and definability counterparts. However, we prefer \emph{primitively recursive} to \emph{primitively computable}, as the former better indicates the definition both referring to.

For more notions in computability and reverse mathematics, we refer the reader to Soare \cite{Soare:1987.book} and Simpson \cite{Simpson:2009.SOSOA}. We also need some knowledge in algorithmic randomness which can be found in Downey and Hirschfeldt \cite{Downey.Hirschfeldt:2010.book}. Furthermore, we recommend a recent survey paper by Hirschfeldt \cite{Hirschfeldt:slicing} for a general picture of the reverse mathematics of Ramsey theory.

\section{Preparations}\label{s:Preparations}

Our center concept is formulated below.

\begin{definition}\label{def:presv-arith}
A set $Y$ \emph{preserves properly $\Xi$-definitions} (relative to $X$) for $\Xi$ among $\Delta^0_{n+1}, \Pi^0_n, \Sigma^0_n$ where $n > 0$, if every properly $\Xi$ (relative to $X$) set is properly $\Xi$ relative to $Y$ ($X \oplus Y$). $Y$ \emph{preserves the arithmetic hierarchy} (relative to $X$) if $Y$ preserves $\Xi$-definitions (relative to $X$) for all $\Xi$ among $\Delta^0_{n+1}, \Pi^0_n, \Sigma^0_n$ where $n > 0$.

Suppose that $\Phi = (\forall X) (\exists Y) \varphi(X,Y)$ and $\varphi$ is arithmetic. $\Phi$ \emph{admits preservation of properly $\Xi$-definitions} if for each $Z$ and $X \leq_T Z$ there exists $Y$ such that $Y$ preserves properly $\Xi$-definitions relative to $Z$ and $\varphi(X,Y)$ holds. $\Phi$ \emph{admits preservation of the arithmetic hierarchy} if for each $Z$ and $X \leq_T Z$ there exists $Y$ such that $Y$ preserves the arithmetic hierarchy relative to $Z$ and $\varphi(X,Y)$ holds.
\end{definition}

As $\Xi$-definitions relative to $X$ are trivially $\Xi$ relative to $X \oplus Y$, usually we omit the adverb \emph{properly} in the above definition and simply say that $Y$ preserves $\Xi$-definitions, etc.

If $\Phi$ admits preservation of $\Xi$-definitions then solving $\Phi$-instances does not simply a properly $\Xi$-definition. Thus we may classify $\Phi$ as a weak proposition. So the above definition captures our motivation in \S 1. In the remaining part of this section, we prove some propositions which will help us in proving preservation and non-preservation results later.

The first proposition slightly simplies Definition \ref{def:presv-arith}.

\begin{proposition}\label{pro:presv-equiv}
Suppose that $n > 0$.
\begin{enumerate}
    \item A set $Y$ preserves $\Sigma^0_n$-definitions relative to $X$ if and only if $Y$ preserves $\Pi^0_n$-definitions relative to $X$.
    \item If $Y$ preserves $\Delta^0_{n+1}$-definitions relative to $X$ then $Y$ preserves $\Pi^0_n$-definitions relative to $X$.
    \item A set $Y$ preserves $\Delta^0_{n+1}$-definitions relative to $X$ if and only if $\Delta^X_{n+1} - \Sigma^X_{n} \subseteq \Delta^{X \oplus Y}_{n+1} - \Sigma^{X \oplus Y}_{n}$.
    \item Let $\Phi = (\forall X) (\exists Y) \varphi(X,Y)$ with $\varphi$ being arithmetic. If for each $X \leq_T Z$ and for every sequence $(A_i: i < \omega)$ with no $A_i$ being $\Sigma^Z_n$ there exists $Y$ such that $\varphi(X,Y)$ holds and no $A_i$ is $\Sigma^{Z \oplus Y}_n$, then $\Phi$ admits simultaneous preservation of $\Sigma^0_n$-, $\Pi^0_n$- and $\Delta^0_{n+1}$-definitions.
\end{enumerate}
\end{proposition}

\begin{proof}
(1) follows trivially from that the complement of a $\Pi^X_n$ set is a $\Sigma^X_n$ set.

For (2), suppose that $A \in (\Pi^X_n - \Delta^X_n) \cap \Delta^{X \oplus Y}_n$. Then $A \leq_T (X \oplus Y)^{(n-1)}$ and $A \oplus X^{(n-1)}$ is of properly computably enumerable degree relative to $X^{(n-1)}$. By relativizing a construction of Shore (see \cite[VI.3.9]{Soare:1987.book} or \cite[Theorem 8.21.15]{Downey.Hirschfeldt:2010.book}), there exists $G \leq_T A \oplus X^{(n-1)}$ which is $1$-generic relative to $X^{(n-1)}$ and thus $n$-generic relative to $X$. So $G$ is properly $\Delta^X_{n+1}$. But
$$
    G \leq_T A \oplus X^{(n-1)} \leq_T (X \oplus Y)^{(n-1)}.
$$
Hence $G$ is $\Delta^0_{n}$ in $X \oplus Y$ and witnesses that $Y$ does not preserve $\Delta^0_{n+1}$-definitions relative to $X$.

For the only-if part of (3), suppose that $Y$ preserves $\Delta^0_{n+1}$-definitions relative to $X$. Fix an arbitrary $A \in \Delta^X_{n+1} - \Sigma^X_n$. Then either $A$ is properly $\Delta^X_{n+1}$ or $A$ is properly $\Pi^X_n$. In the former case $A$ is properly $\Delta^{X \oplus Y}_{n+1}$, while in the latter $A$ is properly $\Pi^{X \oplus Y}_n$ by (2). So in either case, $A \in \Delta^{X \oplus Y}_{n+1} - \Sigma^{X \oplus Y}_n$. For the if part, suppose that $\Delta^X_{n+1} - \Sigma^X_{n} \subseteq \Delta^{X \oplus Y}_{n+1} - \Sigma^{X \oplus Y}_{n}$ and $A \in \Delta^X_{n+1} - \Pi^X_n$. Then $\omega - A \in \Delta^X_{n+1} - \Sigma^X_{n} \subseteq \Delta^{X \oplus Y}_{n+1} - \Sigma^{X \oplus Y}_{n}$. Thus $A \in \Delta^{X \oplus Y}_{n+1} - \Pi^{X \oplus Y}_{n}$. So $Y$ preserves $\Delta^0_{n+1}$-definitions relative to $X$.

(4) follows from (1) and (3) and that there are only countable many arithmetic sets.
\end{proof}

Note that the converse of Proposition \ref{pro:presv-equiv}(2) does not hold. For example, a $\Delta^0_2$ $1$-generic $G$ does not preserve $\Delta^0_2$-definitions but preserves $\Pi^0_1$-definitions since the only computably enumerable sets computable in $G$ is the computable sets. Furthermore, every low set preserves $\Pi^0_2$-definitions. Thus non-preservation of $\Delta^0_2$-definitions does not imply non-preservation of $\Pi^0_2$-definitions.

In the light of Proposition \ref{pro:presv-equiv}(4), people may suggest to introduce a notion like \emph{preservation of non-computable-enumerability} which sounds stronger than \emph{preservation of $\Delta^0_2$ sets}. However, preservation of non-computable-enumerability literally implies that every non-computably-enumerable set is non-computably-enumerable relative to $Y$. But a non-computable $Y$ with such a property cannot be computably enumerable and thus only computable sets preserve non-computable-enumerability. On the other hand, preservation of the arithmetic hierarchy does not seem admitting an alternative like preservation of non-computable-enumerability.

Next we present a proposition that connects definability strength and computability strength. If $\varphi$ is arithmetic and $\Phi = (\forall X)(\exists Y) \varphi(X,Y)$ and if for every $Z$ and $X \leq_T Z$ and every countable sequence $(A_i: i < \omega)$ of sets not computable in $Z$ there exists $Y$ such that $\varphi(X,Y)$ and $A_i \not\leq_T Z \oplus Y$ for each $i$, then we say that \emph{$\Phi$ admits simultaneous avoidance of countably many cones}.

\begin{proposition}\label{pro:sim-avoidance-presv-P1}
If $\Phi$ is a $\Pi^1_2$ sentence that admits simultaneous avoidance of countably many cones then $\Phi$ admits preservation of $\Sigma^0_1$- and $\Pi^0_1$-definitions.
\end{proposition}

\begin{proof}
By Proposition \ref{pro:presv-equiv}, it suffices to prove the preservation of $\Pi^0_1$-definitions. Suppose that $\Phi = (\forall X)(\exists Y) \varphi$ and $\varphi$ is arithmetic. Fix $Z$, $X \leq_T Z$ and $(A_i: i < \omega)$ such that each $A_i$ is properly $\Pi^Z_1$. Let $Y$ be such that $\varphi(X,Y)$ and $A_i \not\leq_T Z \oplus Y$. Then every $A_i$ is properly $\Pi^0_1$ in $Z \oplus Y$.
\end{proof}

The last proposition of this section shows us how the analysis of definability strength leads to reverse mathematics consequences. If $\mathcal{S}$ is a subset of $2^\omega$ and $\Xi$ is among $\Delta^0_{n+1}, \Pi^0_n$ and $\Sigma^0_n$, then we write $\Xi(\mathcal{S})$ for the set $\{A: (\exists X \in \mathcal{S})(A \in \Xi^X)\}$. We say that \emph{$\mathcal{S}$ preserves $\Xi$-definitions (relative to $Z$)} if every properly $\Xi$ (relative to $Z$) set is properly $\Xi(\mathcal{S})$ ($\Xi(Z \oplus \mathcal{S})$ where $Z \oplus \mathcal{S} = \{Z \oplus X: X \in \mathcal{S}\}$), and \emph{$\mathcal{S}$ preserves the arithmetic hierarchy (relative to $Z$)} if $\mathcal{S}$ preserves $\Xi$-definitions (relative to $Z$) for every $\Xi$ among $\Delta^0_{n+1}, \Pi^0_n, \Sigma^0_n$ where $n > 0$. These notions can be naturally extended to $\omega$-models.

\begin{proposition}\label{pro:sep-by-presv}
Suppose that $(\Phi_i: i < \omega)$ and $\Psi$ are true $\Pi^1_2$ sentences and $\Xi$ is among $\Delta^0_{n+1}, \Pi^0_n, \Sigma^0_n$ where $n > 0$.
\begin{enumerate}
    \item All $\Phi_i$ admit preservation of $\Xi$-definitions, if and only if for each $Z$ there exists an $\omega$-model $(\omega, \mathcal{S})$ of $\RCA$ and $\bigwedge_i \Phi_i$ which contains $Z$ and preserves $\Xi$-definitions relative to $Z$.
    \item If $\Psi$ does not admit preservation of $\Xi$-definitions, then there exists $Z$ such that every $\omega$-model $(\omega, \mathcal{S})$ of $\RCA + \Psi$ containing $Z$ does not preserve $\Xi$-definitions relative to $Z$.
    \item If every $\Phi_i$ admits preservation of $\Xi$-definitions but $\Psi$ does not then
    $$
        \RCA + \bigwedge_i \Phi_i \not\vdash \Psi.
    $$
    Conversely, if $\RCA + \bigwedge_i \Phi_i \vdash \Psi$ and $\Psi$ does not admit preservation of $\Xi$-definitions, then some $\Phi_i$ does not admit preservation of $\Xi$-definitions either.
\end{enumerate}
\end{proposition}

\begin{proof}
(1) Suppose that $\Phi_i = (\forall X)(\exists Y) \varphi_i(X,Y)$. 

For an arbitrary $Z$, if we have an $\omega$-model $(\omega, \mathcal{S})$ as described then for every $\Phi_i$ and every $X \leq_T Z$ we can pick $Y \in \mathcal{S}$ such that $\varphi_i(X,Y)$ holds and $Y$ preserves $\Xi$-definitions relative to $Z$. So every $\Phi_i$ admits preservation of $\Xi$-definitions.

Conversely, suppose that every $\Phi_i$ admits preservation of $\Xi$-definitions, then we can build a sequence $((\omega, \mathcal{S}_n): n < \omega)$ such that
\begin{itemize}
    \item each $\mathcal{S}_n$ is of the form $\{Y: Y \leq_T X_n\}$ for some $X_n$ and $X_0 = Z$;
    \item $X_n \leq_T X_{n+1}$ and $X_{n+1}$ preserves $\Xi$-definitions relative to $X_n$;
    \item for each $i$ and $X \in \mathcal{S}_n$ there exist $m > n$ and $Y \in \mathcal{S}_m$ such that $\varphi_i(X,Y)$.
\end{itemize}
Let $\mathcal{S} = \bigcup_n \mathcal{S}_n$. Then $(\omega, \mathcal{S})$ is as desired.

(2) Suppose that $\Psi = (\forall X)(\exists Y) \psi(X,Y)$ does not admit preservation of $\Xi$-definitions. Then there exist $Z$ and $X \leq_T Z$ such that if $\psi(X,Y)$ then $Y$ does not preserve $\Xi$-definitions relative to $Z$. Hence any $(\omega, \mathcal{S}) \models \RCA + \Psi$ containing $Z$ is a desired model.

(3) follows directly from (1) and (2).
\end{proof}

\section{Preservations}\label{s:presv}

Recall that our main purpose is to analyze the definability strength of propositions in Ramsey theory. Here we show that some $\Pi^1_2$ propositions are weak in terms of definability strength. We also show that some computability notions related to the reverse mathematics of Ramsey theory are also weak.

\subsection{Cohen genericity and randomness}

It is not surprising that both Cohen generic and random reals are weak in terms of definability strength.

\begin{proposition}\label{pro:generic-presv}
If $G$ is sufficiently Cohen generic relative to $X$ then $G$ preserves the arithmetic hierarchy relative to $X$.
\end{proposition}

\begin{proof}
We use $\Vdash$ for Cohen forcing. Note that $\{\sigma \in 2^{<\omega}: \sigma \Vdash \varphi\}$ is $\Sigma^X_n$ if $\varphi$ is $\Sigma^X_n$.

Suppose that $A \not\in \Sigma^X_n$ and $\varphi$ is $\Sigma^X_n$ where $n > 0$. We claim that the following set is meager:
$$
    \mathcal{S} = \{Y: (\forall i) (i \in A \leftrightarrow \varphi(Y, i))\}.
$$
Otherwise, $\mathcal{S}$ is comeager in $\{Y: \sigma \prec Y\}$ for some $\sigma \in 2^{<\omega}$. So, $A = \{i: \sigma \Vdash \varphi(G,i)\}$. But then $A$ is $\Sigma^X_n$. Hence if $G$ is sufficiently Cohen generic relative to $X$ then $A \not\in \Sigma^{X \oplus G}_n$. 

By Proposition \ref{pro:presv-equiv}, $G$ preserves the arithmetic hierarchy relative to $X$.
\end{proof}

If we carefully examine the above proof then we can obtain the following finer result. Recall that $G$ is \emph{weakly $n$-generic} relative to $X$ if $G$ meets every $\Sigma^X_n$ dense open set of Cantor space (see \cite[\S 2.24]{Downey.Hirschfeldt:2010.book}).

\begin{corollary}\label{cor:generic-presv}
If $G$ is weakly $(n+1)$-generic relative to $X$ then $G$ simultaneously preserves $\Pi^0_n$-, $\Sigma^0_n$- and $\Delta^0_{n+1}$-definitions relative to $X$.
\end{corollary}

\begin{proof}
Fix $G$ being weakly $(n+1)$-generic relative to $X$. By Proposition \ref{pro:presv-equiv}, it suffices to show that $\Delta^X_{n+1} - \Sigma^X_n \subseteq \Delta^{X \oplus G}_{n+1} - \Sigma^{X \oplus G}_n$. Let $A \in \Delta^X_{n+1} - \Sigma^X_n$ and $\varphi$ be $\Sigma^X_n$. By the proof of Proposition \ref{pro:generic-presv}, the following set is dense and $\Sigma^X_{n+1}$:
$$
  D = \{\sigma: (\exists i)((i \in A \wedge \sigma \Vdash \neg \varphi(\dot{G}, i)) \vee (i \not\in A \wedge \sigma \Vdash \varphi(\dot{G}, i)))\}.
$$
So $G$ meets $D$ and $A \neq \{i: \varphi(G, i)\}$.
\end{proof}

We apply Proposition \ref{pro:generic-presv} to a $\Pi^1_2$ statement which connects the existence of Cohen generic and the reverse mathematics of model theory. Introduced by Hirschfeldt, Shore and Slaman \cite{Hirschfeldt.Shore.ea:2009}, $\Pi^0_1 \operatorname{G}$ asserts that for every uniformly $\Pi^0_1$ sequence $(D_n: n < \omega)$ of dense open sets in Cantor space there exists $G \in \bigcap_n D_n$. It follows immediately from Proposition \ref{pro:generic-presv} that:

\begin{corollary}\label{cor:PiG-presv}
$\Pi^0_1 \operatorname{G}$ admits preservation of the arithmetic hierarchy.
\end{corollary}

For the next preservation result concerning random reals, we identify $\mathbb{R}$ with Cantor space and denote Lebesgue measure by $m$. Recall that a real is sufficiently random if it avoids sufficiently many Lebesgue null sets.

\begin{proposition}\label{pro:random-presv}
If $R$ is sufficiently random relative to $X$ then $R$ preserves the arithmetic hierarchy relative to $X$.
\end{proposition}

\begin{proof}
Suppose that $A \not\in \Sigma^X_n$ and $\varphi$ is $\Sigma^X_n$ where $n > 0$. We claim that the following set is Lebesgue null:
$$
    \mathcal{S} = \{Y: (\forall i)(i \in A \leftrightarrow \varphi(Y, i))\}.
$$
Otherwise there exists $\sigma \in 2^{<\omega}$ such that
$$
    m \{Y \in \mathcal{S}: \sigma \prec Y\} > 2^{-|\sigma|-1}.
$$
Then
$$
    A = \{i: m \{Y: \sigma \prec Y \wedge \varphi(Y; i)\} > 2^{-|\sigma|-1}\},
$$
and $A$ is $\Sigma^X_n$ by Kurtz \cite[Lemma 2.1a]{Kurtz:81.thesis} (see also \cite[Lemma 6.8.1]{Downey.Hirschfeldt:2010.book}), contradicting that $A \not\in \Sigma^X_n$. Hence if $R$ is sufficiently random relative to $X$ then $A \not\in \Sigma^{X \oplus R}_n$. 

By Proposition \ref{pro:presv-equiv}, $R$ preserves the arithmetic hierarchy relative to $X$.
\end{proof}

Similarly, if we examine the effectiveness of the above proof then we can obtain the finer preservation result below. Recall that a real is \emph{weakly $n$-random} if it avoids every $\Pi^0_n$ null set (see \cite[\S 7.2]{Downey.Hirschfeldt:2010.book}).

\begin{corollary}\label{cor:random-presv}
If $R$ is weakly $(n+1)$-random relative to $X$ then $R$ simultaneously preserves $\Sigma^0_n$-, $\Pi^0_n$- and $\Delta^0_n$-definitions relative to $X$.
\end{corollary}

\begin{proof}
Fix $R$ being weakly $(n+1)$-random relative to $X$. As in the proof of Corollary \ref{cor:generic-presv}, it suffices to show that $\Delta^X_{n+1} - \Sigma^X_n \subseteq \Delta^{X \oplus R}_{n+1} - \Sigma^{X \oplus R}_n$. Let $A \in \Delta^X_{n+1} - \Sigma^X_n$ and $\varphi$ be $\Sigma^X_n$. By the proof of Proposition \ref{pro:random-presv}, the following set is $\Pi^X_{n+1}$ and null:
$$
    \mathcal{S} = \{Y: (\forall i)(i \in A \leftrightarrow \varphi(Y, i))\}.
$$
So $R \not\in \mathcal{S}$ and $A \neq \{i: \varphi(R, i)\}$.
\end{proof}

From the short proofs above, the reader may have found that the definability of forcing is a key to preservation.

\subsection{Weak K\"{o}nig's Lemma}

We devote this section to the proof of the following preservation theorem for $\WKL$.

\begin{theorem}\label{thm:WKL-presv}
$\WKL$ admits preservation of the arithmetic hierarchy.
\end{theorem}

As in the last subsection, we prove the above theorem by proving the definability of a forcing. We fix an $X$ and an $X$-computable infinite binary tree $T_0$ and need to find $G \in [T_0]$ preserving the arithmetic hierarchy relative to $X$. As our proof will be relativizable, we may assume that $X$ is computable. To build $G$, we build a sequence of computable trees $(T_i: i < \omega)$ such that each $T_{i+1}$ is a subtree of $T_i$, and then obtain $G$ as a member of $\bigcap_i [T_i]$. We introduce a forcing notion and require that each $T_i$ is a condition forcing a fragment of the preservation requirement.

\begin{definition}
If $T$ is a computable subtree of $2^{<\omega}$, then a \emph{primitively recursive subtree} $S$ of $T$ is a tree of the form
$$
    S = T \cap R
$$
where $R$ is a primitively recursive subset of $2^{<\omega}$. Let $\operatorname{Pr}(T)$ denote the set of all primitively recursive subtrees of $T$. Note that a tree in $\operatorname{Pr}(T)$ could be finite.

Let $\mathbb{P}$ be the set of infinite primitively recursive subtrees of $T_0$. A tree in $\mathbb{P}$ is a \emph{forcing condition}. For $S$ and $T$ in $\mathbb{P}$, $S \leq T$ if and only if $S$ is a subtree of $T$.
\end{definition}

The following lemma shows that moving from computable trees to primitively recursive trees costs nothing as long as we concern only the set of infinite paths of a tree.

\begin{lemma}\label{lem:Pi01-forcing-rec-subtrees}
For every computable tree $S$, there exists a primitively recursive tree $T$ such that $[S] = [T]$.
\end{lemma}

\begin{proof}
Fix a computable tree $S$. Let $\psi(\sigma; \vec{x})$ be a $\Sigma^0_0$ formula such that
$$
    S = \{\sigma: (\forall \vec{x}) \neg \psi(\sigma, \vec{x})\}.
$$
Let
$$
    T = \{\sigma: (\forall \vec{x}) (\max \vec{x} < |\sigma| \to \neg \psi(\sigma, \vec{x}))\}.
$$
Then $T$ is a primitively recursive tree and $[S] = [T]$.
\end{proof}

So we assume that every computable tree appearing below is primitively recursive. But the main advantage of using primitively recursive trees is that we can better control the complexity.

\begin{lemma}\label{lem:Pi01-forcing-Pr}
Suppose that $T \in \mathbb{P}$. Then $\operatorname{Pr}(T) - \mathbb{P}$ can be identified with a $\Sigma^0_1$ set.
\end{lemma}

\begin{proof}
Fix a computable enumeration $(S_n: n < \omega)$ of $\operatorname{Pr}(T)$.
\begin{align*}
    S_n \not\in \mathbb{P} \Leftrightarrow S_n \text{ is finite} \Leftrightarrow (\exists m) (\forall \sigma \in 2^m) (\sigma \not\in S_n).
\end{align*}
So $\operatorname{Pr}(T) - \mathbb{P}$ can be identified with the $\Sigma^0_1$ set $\{n: S_n \text{ is finite}\}$.
\end{proof}

We define the forcing relation below. Our forcing language is the first order language of arithmetic augmented by a unary predicate $G$.

\begin{definition}\label{def:Pi01-forcing-relation}
Let $T \in \mathbb{P}$.
\begin{enumerate}
    \item If $\psi(G, \vec{x})$ is $\Sigma^0_0$, then
        $$
            T \Vdash (\exists \vec{x}) \psi \Leftrightarrow (\exists n) (\forall \sigma \in 2^n \cap T) (\exists \vec{m}) (\max \vec{m} < n \wedge \psi(\sigma, \vec{x}) [\vec{m}])
        $$
        and
        $$
	    T \Vdash \neg (\exists \vec{x}) \psi \Leftrightarrow (\forall \sigma \in T) (\forall \vec{n}) (\max \vec{n} < |\sigma| \to \neg \psi(\sigma, \vec{x})[\vec{n}]).
        $$
    \item Suppose that $\varphi(G)$ is not of the above forms.
    \begin{enumerate}
	\item If $\varphi(G)$ is of the form $(\exists x) \psi(G; x)$, then
        $$
            T \Vdash \varphi(G) \Leftrightarrow (\exists n) (T \Vdash \psi(G, x)[n]).
        $$
        \item If $\varphi(G)$ is of the form $\psi \wedge \theta$ then
        $$
	    T \Vdash \varphi(G) \Leftrightarrow T \Vdash \psi \text{ and } T \Vdash \theta.
        $$
	\item If $\varphi(G)$ is of the form $\neg \psi(G)$, then
        $$
            T \Vdash \varphi(G) \Leftrightarrow S \not\Vdash \psi(G) \text{ for all } S \leq T,
        $$
        where $S \not\Vdash \psi(G)$ means $\neg (S \Vdash \psi(G))$.
    \end{enumerate}
\end{enumerate}
\end{definition}

Definition \ref{def:Pi01-forcing-relation} appears slightly different from usual forcing definitions. The difference is introduced for definability purpose. But the next two lemmata show that this difference is superficial.

\begin{lemma}\label{lem:Pi01-forcing-density}
For every arithmetic formula $\varphi$, the following set is dense
$$
    \{T \in \mathbb{P}: T \Vdash \varphi \text{ or } T \Vdash \neg \varphi\}.
$$
\end{lemma}

\begin{proof}
Fix a formula $\varphi$ and a tree $T \in \mathbb{P}$.

Suppose that $\varphi$ is of the form $(\exists \vec{x}) \psi$ and $\psi$ is $\Sigma^0_0$. Let
$$
    S = \{\sigma \in T: (\forall \vec{n}) (\max \vec{n} < |\sigma| \to \neg \psi(\sigma; \vec{x})[\vec{n}])\}.
$$
Clearly $S$ is a computable subtree of $T$. If $S$ is infinite then $S \Vdash \neg \varphi$. If $S$ is finite then pick $\sigma \in T - S$ such that the following set is infinite:
$$
    R = \{\tau \in T: \tau \text{ is comparable with } \sigma\}.
$$
Then $R \in \mathbb{P}$ and $R \leq T$. As $\sigma \in T - S$, $\psi(\sigma; \vec{x})[\vec{n}]$ for some $\vec{n}$ with $\max \vec{n} < |\sigma|$. So $R \Vdash \varphi$.

If $\varphi$ is not $\Sigma^0_1$ then the density follows from clause (2c) of Definition \ref{def:Pi01-forcing-relation}.
\end{proof}

If $\mathcal{F}$ is a sufficiently $\mathbb{P}$-generic filter then $\bigcap_{T \in \mathcal{F}} [T]$ contains exactly one real $G$. Conversely, from each real $H$ we can define an induced filter $\mathcal{F}(H)$ over $\mathbb{P}$ as the trees in $\mathbb{P}$ having $H$ as an infinite path. Moreover, if $G$ is the real defined from a generic filter $\mathcal{F}$ as above then $\mathcal{F} = \mathcal{F}(G)$. So we may say that a real $G$ is sufficiently generic.

\begin{lemma}\label{lem:Pi01-forcing-completeness}
Suppose that $G$ is a sufficiently $\mathbb{P}$-generic. Then for each $n > 0$ and each $\Sigma^0_n$ ($\Pi^0_n$) formula $\varphi(G)$, $\varphi(G)$ holds if and only if $T \Vdash \varphi(G)$ for some $T \in \mathcal{F}(G)$.
\end{lemma}

\begin{proof}
Suppose that $\psi$ is $\Sigma^0_0$ and $\varphi = (\exists \vec{x}) \psi$. If $\varphi(G)$ holds then $\psi(G; \vec{x})[\vec{n}]$ holds for some $\vec{n}$. Let
$$
    T = \{\sigma \in T_0: |\sigma| \leq \max \vec{n} \vee \psi(\sigma; \vec{x})[\vec{n}]\}.
$$
Then $G \in [T]$ and $T \Vdash \varphi$. Conversely, if $T \in \mathcal{F}(G)$ and $T \Vdash \varphi$ then there exist $n$ and $\sigma \in 2^n \cap T$ such that $\sigma \prec G$ and $\psi(\sigma; \vec{x})[\vec{m}]$ for some $\vec{m}$. Hence $\varphi(G)$ holds.

Suppose that $\psi$ is $\Sigma^0_0$ and $\varphi = \neg (\exists \vec{x}) \psi$. If $\varphi(G)$ holds then let
$$
    T = \{\sigma \in T_0: (\forall \vec{n}) (\max \vec{n} < |\sigma| \to \neg \psi(\sigma; \vec{x})[\vec{n}])\}.
$$
Then $T \in \mathcal{F}(G)$ and $T \Vdash \varphi$. Conversely, if $T \in \mathcal{F}(G)$ and $T \Vdash \varphi$ then $\neg \psi(\sigma; \vec{x})[\vec{n}]$ for every $\sigma \prec G$ and every $\vec{n}$ with $\max \vec{n} < |\sigma|$. Hence $\varphi(G)$ holds.

Suppose that $\varphi$ is not in above forms. If $\varphi$ is of the form $(\exists x) \psi$ then
\begin{align*}
    \varphi(G) \text{ holds } & \Leftrightarrow (\exists n) (\psi(G; x) [n] \text{ holds }) \\
    & \Leftrightarrow (\exists n) (\exists T \in \mathcal{F}(G)) (T \Vdash \psi(G; x) [n]) \\
    & \Leftrightarrow (\exists T \in \mathcal{F}(G)) (T \Vdash \varphi(G)),
\end{align*}
where the second equivalence is by the induction hypothesis and the last is by clause (2a) of Definition \ref{def:Pi01-forcing-relation}.

If $\varphi$ is of the form $\psi \wedge \theta$ then
\begin{align*}
    \varphi(G) \text{ holds } & \Leftrightarrow \psi(G) \text{ and } \theta(G) \text{ hold} \\
    & \Leftrightarrow (\exists S, T \in \mathcal{F}(G)) (S \Vdash \psi \text{ and } T \Vdash \theta) \\
    & \Leftrightarrow (\exists R \in \mathcal{F}(G)) (R \Vdash \varphi(G)),
\end{align*}
where the second equivalence is by the induction hypothesis and the last is by taking $R = S \cap T$ and clause (2b) of Definition \ref{def:presv-arith}.

If $\varphi$ is of the form $\neg \psi$ then
\begin{align*}
    \varphi(G) \text{ holds } & \Leftrightarrow \psi(G) \text{ does not hold} \\
    & \Leftrightarrow (\forall T \in \mathcal{F}(G)) (T \not\Vdash \psi)
\end{align*}
By Lemma \ref{lem:Pi01-forcing-density} and the genericity of $G$, the last statement above is equivalent to the following
$$
    T \Vdash \neg \psi \text{ for some } T \in \mathcal{F}(G).
$$
\end{proof}

Next we show that our forcing relation is definable.

\begin{lemma}\label{lem:Pi01-forcing-definability}
If $n > 0$ and $\varphi(G)$ is a $\Sigma^0_n$ ($\Pi^0_n$) formula in prenex normal form then $T \Vdash \varphi(G)$ is $\Sigma^0_n$ ($\Pi^0_n$) predicate for $T \in \mathbb{P}$.
\end{lemma}

\begin{proof}
The case where $n = 1$ follows from clause (1) of Definition \ref{def:Pi01-forcing-relation}.

Assume that $n > 1$. For $\varphi(G) \in \Sigma^0_n$, the lemma follows from the induction hypothesis and Definition \ref{def:Pi01-forcing-relation}(2). Suppose that $\varphi(G)$ is $\Pi^0_n$. Then $\varphi(G)$ is of the form $\neg \psi(G)$ for some $\psi(G) \in \Sigma^0_{n}$. So,
\begin{align*}
    T \Vdash \varphi(G) &\Leftrightarrow S \not\Vdash \psi(G) \text{ for all } S \leq T \\
    &\Leftrightarrow (\forall S \in \operatorname{Pr}(T))(S \in \mathbb{P} \to S \not\Vdash \psi(G)).
\end{align*}
By the induction hypothesis, $S \not\Vdash \psi(G)$ is a $\Pi^0_n$ predicate of $S$. So the last statement above is a $\Pi^0_n$ predicate of $T$ as $n > 1$ and $\operatorname{Pr}(T) - \mathbb{P}$ is $\Sigma^0_1$ by Lemma \ref{lem:Pi01-forcing-Pr}.
\end{proof}

Now we can show that the definability strength of $\mathbb{P}$-generic reals is weak.

\begin{lemma}\label{lem:Pi01-forcing-presv}
If $A \not\in \Sigma^0_n$ ($\Pi^0_n$) and $\varphi(G, x) \in \Sigma^0_n$ ($\Pi^0_n$) where $n > 0$, then the set of $T \in \mathbb{P}$ satisfying the following property is dense:
$$
    (\exists n \in A) (T \Vdash \neg \varphi(G; x)[n]) \vee (\exists n \not\in A) (T \Vdash \varphi(G; x)[n]).
$$
\end{lemma}

\begin{proof}
Fix $T \in \mathbb{P}$. We need to find $S \leq T$ with
$$
    (\exists n \in A) (S \Vdash \neg \varphi(G, x)[n]) \vee (\exists n \not\in A) (S \Vdash \varphi(G, x)[n]).
$$

Firstly, assume that $A \not\in \Sigma^0_1$ and $\varphi(G, x) \in \Sigma^0_1$. Let
$$
    W = \{n: (\exists l)(\forall \sigma \in T \cap 2^l) \varphi(\sigma, x)[n]\}.
$$
Then $W \in \Sigma^0_1$ and thus $W \neq A$. Fix $n \in W \bigtriangleup A$. If $n \in W - A$, then let $\sigma \in T$ be such that $\varphi(\sigma, n)$ and the following subtree of $T$ is infinite:
$$
    S = \{\tau \in T: \tau \text{ is comparable with } \sigma\}.
$$
Then $S \leq T$ and $S \Vdash \varphi(G, x)[n]$. Suppose that $n \in A - W$. Then let
$$
    S = \{\sigma \in T: \neg \varphi(\sigma, x)[n]\}.
$$
By Lemma \ref{lem:Pi01-forcing-rec-subtrees}, we may assume that $S \in \mathbb{P}$. Then $S \Vdash \neg \varphi(G, x)[n]$.

Assume that $n > 1$, $A \not\in \Sigma^0_n$ and $\varphi(G, x) \in \Sigma^0_n$. Let
$$
    U = \{n: (\exists S \leq T)(S \Vdash \varphi(G, x)[n])\}.
$$
By Lemmata \ref{lem:Pi01-forcing-Pr} and \ref{lem:Pi01-forcing-definability}, $U \in \Sigma^0_n$ and thus $U \neq A$. Fix $n \in A \bigtriangleup U$. If $n \in A - U$, then $S \not\Vdash \varphi(G, x)[n]$ for all $S \leq T$. By Definition \ref{def:Pi01-forcing-relation}, $T \Vdash \neg \varphi(G, x)[n]$. So we can simply let $S = T$. If $n \in U - A$, then let $S \leq T$ be such that $S \Vdash \varphi(G, x)[n]$.

Finally, assume that $n > 0$, $A \not\in \Pi^0_n$ and $\varphi(G, x) \in \Pi^0_n$. Then $\omega - A \not\in \Sigma^0_n$ and $\varphi(G, x)$ is of the form $\neg \psi(G, x)$ for some $\psi(G, x) \in \Sigma^0_n$. By the proof above, the following set is dense
$$
    \{T \in \mathbb{P}: (\exists n \in \omega - A) (T \Vdash \neg \psi(G, x)[n]) \vee (\exists n \not\in \omega - A) (T \Vdash \psi(G, x)[n])\}.
$$
Note that if $T \Vdash \psi(G, x)[n]$ then $T \Vdash \neg \neg \psi(G, x)[n]$. So the above set is contained in the following set
$$
    \{T \in \mathbb{P}: (\exists n \not\in A) (T \Vdash \varphi(G, x)[n]) \vee (\exists n \in A) (T \Vdash \neg \varphi(G, x)[n])\}.
$$
This proves the lemma for the last case.
\end{proof}

Theorem \ref{thm:WKL-presv} follows from Lemmata \ref{lem:Pi01-forcing-completeness} and \ref{lem:Pi01-forcing-presv}.

\subsection{$\COH$}

Recall that a cohesive set for a sequence $(A_n: n < \omega)$ is an infinite set $C$ such that either $C \subseteq^* A_n$ or $C \subseteq^* \omega - A_n$ for each $n$. $\COH$ is the assertion that every sequence of sets has a cohesive set. Cholak et alia \cite{Cholak.Jockusch.ea:2001.Ramsey} introduced $\COH$ and proved that it is a consequence of $\RCA + \RT^2_2$. Here we prove that $\COH$ enjoys a preservation property.

\begin{theorem}\label{thm:COH-presv-D2}
If $\vec{A} = (A_i: i < \omega)$ be $Z$-computable and $(B_i: i < \omega)$ are such that no $B_i$ is $\Sigma^Z_1$, then there exists an $\vec{A}$-cohesive set $C$ such that no $B_i$ is $\Sigma^0_1$ relative to $Z \oplus C$.

Hence $\COH$ admits preservation of $\Delta^0_2$-definitions.
\end{theorem}

The above theorem follows immediately from a property of Mathias forcing. Recall that a \emph{Mathias condition} is a pair $(\sigma,X) \in [\omega]^{<\omega} \times [\omega]^\omega$ such that $\max \sigma < \min X$. We identify a Mathias condition $(\sigma,X)$ with the set below:
$$
    \{Y \in [\omega]^\omega: \sigma \subset Y \subseteq \sigma \cup X\}.
$$
For two Mathias conditions $(\sigma,X)$ and $(\tau,Y)$, $(\tau,Y) \leq_M (\sigma,X)$ if and only if $(\tau,Y) \subseteq (\sigma,X)$ under the above convention.

\begin{lemma}\label{lem:Mathias-prs-non-re}
Let $B$ be a set and $(\sigma, X)$ be a Mathias condition such that $X \leq_T Z$ and $B \not\in \Sigma^Z_1$. Then for every $e$ there exists $(\tau, Y) \leq_M (\sigma, X)$ such that $Y =^* X$ and $B \neq W^{Z \oplus G}_e$ for all $G \in (\tau, Y)$.
\end{lemma}

\begin{proof}
Let
$$
    W = \{n: (\exists G \in (\sigma, X))(n \in W^{Z \oplus G}_e)\}.
$$
Then $W$ is $\Sigma^Z_1$ and thus $W \neq A$.

Fix $n \in W \bigtriangleup A$. If $n \in W - A$ then let $G \in (\sigma, X)$ and $l > |\sigma|$ be such that $n \in W^{Z \uh l \oplus G \uh l}_e$ and let $(\tau, Y) = (G \uh l, X \cap (\max \tau, \infty))$. If $n \in A - W$ then let $(\tau, Y) = (\sigma, X)$.
\end{proof}

\subsection{Erd\H{o}s-Moser Principle}

A \emph{tournament} is a binary relation $R \subseteq \omega \times \omega$ induced by a coloring $c: [\omega]^2 \to 2$ in the following way:
$$
  (c(x,y) = 1 \to x R y) \wedge (c(x,y) = 0 \to y R x).
$$
A set $X$ is transitive for a tournament $R$ if $R \cap X^2$ is a transitive relation. \emph{Erd\H{o}s-Moser Principle} ($\EM$) is the assertion that every tournament admits an infinite transitive set.

\begin{theorem}\label{thm:EM-presv-D2}
$\EM$ admits preservation of $\Delta^0_2$-definitions.
\end{theorem}

With Theorem \ref{thm:COH-presv-D2}, we can reduce the above theorem to a preservation property of a consequence of $\EM$. If $R$ is a tournament such that either $(\forall^\infty y) x R y$ or $(\forall^\infty y) y R x$ for every $x$ then $R$ is \emph{stable}. $\SEM$ is $\EM$ restricted to stable tournaments. Clearly $\RCA + \COH + \SEM \vdash \EM$. So Theorem \ref{thm:EM-presv-D2} follows from Proposition \ref{pro:sep-by-presv}, Theorem \ref{thm:COH-presv-D2} and Lemma \ref{lem:SEM-presv-D2} below.

\begin{lemma}\label{lem:SEM-presv-D2}
For each $Z$ and a $Z$-computable stable tournament $R$ and each sequence $(A_i: i < \omega)$ of sets such that $A_i \not\in \Sigma^Z_1$ for all $i$, there exists an infinite $R$-transitive $G$ such that $A_i \not\in \Sigma^{Z \oplus G}_1$ for all $i$. So $\SEM$ admits preservation of $\Delta^0_2$-definitions.
\end{lemma}

Below we prove the above lemma. Fix $Z$, $R$ and $(A_i: i < \omega)$ as in the statement of Lemma \ref{lem:SEM-presv-D2}. Without loss of generality, assume that $Z$ is computable. We build an infinite $R$-transitive $G$ as desired by Mathias forcing.

Let $f: \omega \to 2$ be as follows:
$$
    f(x) =
    \left\{
      \begin{array}{ll}
        0, & (\forall^\infty y) (x R y); \\
        1, & (\forall^\infty y) (y R x).
      \end{array}
    \right.
$$
By the stability of $R$, $f$ is total.

A Mathias condition $(\sigma, X)$ is \emph{acceptable}, if and only if
\begin{itemize}
    \item[(a1)] for all $x \in \sigma$ and $y \in X$,
        $$
            (f(x) = 0 \to x R y) \wedge (f(x) = 1 \to y R x);
        $$
    \item[(a2)] $\sigma\<y\>$ is $R$-transitive for all $y \in X$.
\end{itemize}

To build $G$, we build a descending sequence of acceptable Mathias conditions. Note that, if $(\sigma, X)$ is $R$-acceptable then
$$
    (\forall x \in X) (\forall a, b \in \sigma\<x\>) (a R b \to f(a) \leq f(b)).
$$
In general, if $g$ is an arbitrary $2$-coloring of $\omega$, then we say that \emph{$R$ and $g$ are compatible on $\sigma \in [\omega]^{<\omega}$} if
$$
    (\forall a, b \in \sigma) (a R b \to g(a) \leq g(b)).
$$

\begin{lemma}\label{lem:compatible-seq}
Let $(\sigma, X)$ be acceptable. If $\tau \in [X]^{<\omega}$ is such that $\sigma\tau$ is $R$-transitive and $R$ and $f$ are compatible on $\sigma\tau$, then $(\sigma\tau, X \cap (n, \infty))$ is acceptable for sufficiently large $n$.
\end{lemma}

\begin{proof}
Let $\tau$ be as above. As $\sigma\tau$ is $R$-transitive, we can list elements of $\sigma\tau$ in $R$-ascending order:
$$
    a_0 R a_1 R \ldots R a_{k-1},
$$
where $k = |\sigma\tau|$. Let
\begin{gather*}
    X_0 = \{x \in X: x > \max \tau \wedge x R a_0\}, \\
    X_i = \{x \in X: x > \max \tau \wedge a_{i-1} R x R a_i\}, 0 < i < k, \\
    X_k = \{x \in X: x > \max \tau \wedge a_{k-1} R x\}.
\end{gather*}
By the stability of $R$ and that $R$ and $f$ are compatible on $\sigma\tau$, $X_i =^* X$ for a unique $i \leq k$. So, $\sigma\tau\<x\>$ is $R$-transitive for all $x \in X_i$. By the definition of $X_i$, for each $a \in \sigma\tau$, if $f(a) = 0$ then $a R x$ for all $x \in X_i$, otherwise $f(a) = 1$ and $x R a$ for all $x \in X_i$. Hence, $(\sigma\tau, X_i)$ is acceptable.
\end{proof}

So, if we can find sequences satisfying the condition of the above lemma then we can extend acceptable conditions.

\begin{lemma}\label{lem:EM-acceptable-longer}
Each acceptable $(\sigma,X)$ can be extended to an acceptable $(\tau, Y)$ such that $|\sigma| < |\tau|$ and $Y =^* X$.
\end{lemma}

\begin{proof}
Let $x = \min X$. By the remark preceding Lemma \ref{lem:compatible-seq}, $R$ and $f$ are compatible on $\sigma\<x\>$. So, the lemma follows from Lemma \ref{lem:compatible-seq}.
\end{proof}

The following density lemma is the key of the proof.

\begin{lemma}\label{lem:EM-acceptable-density}
Let $(\sigma,X)$ be acceptable and $(A_i: i < \omega)$ be a sequence of sets such that $A_i \not\in \Sigma^{X}_1$ for all $i$. Then for every $e$ and every $k$ there exists an acceptable $(\tau,Y) \leq_M (\sigma,X)$ such that $A_i \not\in \Sigma^{Y}_1$ for all $i$ and $A_k \neq W^G_e$ for all $R$-transitive $G \in (\tau, Y)$.
\end{lemma}

\begin{proof}
Let $\mathcal{F}$ be the set of $g: \omega \to 2$ such that $R$ and $g$ are compatible on $\sigma\<x\>$ for all $x \in X$. Then $\mathcal{F}$ can be identified with a $\Pi^{X}_1$ class in Cantor space and $f \in \mathcal{F}$. Let $W$ be the set of $n$ such that for every $g \in \mathcal{F}$ there exists $\xi \in [X]^{<\omega}$ satisfying
$$
     \sigma\xi \text{ is $R$-transitive } \wedge \text{$R$ and $g$ are compatible on $\sigma\xi$} \wedge n \in W^{\sigma\xi}_e.
$$
By the compactness of $\mathcal{F}$, $W \in \Sigma^{X}_1$ and thus $W \neq A_k$. Fix $n \in W \bigtriangleup A_k$.

\emph{Case 1.} $n \in W - A_k$. By the definition of $W$ and that $f \in \mathcal{F}$, we can pick $\xi \in [X]^{<\omega}$ such that $\sigma\xi$ is $R$-transitive, $R$ and $f$ are compatible on $\sigma\xi$ and $n \in W^{\sigma\xi}_e$. Apply Lemma \ref{lem:compatible-seq} to $(\sigma,X)$ and $\xi$, we can obtain a desired extension $(\tau,Y)$ with $\tau = \sigma\xi$ and $Y =^* X$.

\emph{Case 2.} $n \in A_k - W$. Let $\mathcal{G}$ be the set of $g \in \mathcal{F}$ such that
$$
    (\sigma\xi \text{ is $R$-transitive} \wedge \text{$R$ and $g$ are compatible on $\sigma\xi$}) \to n \not\in W^{\sigma\xi}_e
$$
for all $\xi \in [X]^{<\omega}$.
By the preservation property of $\WKL$ (Theorem \ref{thm:WKL-presv}), we can pick $g \in \mathcal{G}$ such that $A_i \not\in \Sigma^{X \oplus g}_1$ for all $i$. If $g^{-1}(0) \cap X$ is infinite then let $Y = g^{-1}(0) \cap X$, otherwise let $Y = g^{-1}(1) \cap X$. Thus, $Y \leq_T X \oplus g$ and $A_i \not\in \Sigma^Y_1$ for all $i$.

\begin{claim}
If $\xi \in [Y]^{<\omega}$ then $R$ and $g$ are compatible on $\sigma\xi$.
\end{claim}

\begin{proof}
It suffices to show that $R$ and $g$ are compatible on arbitrary $(a, b) \in [\sigma\xi]^2$. Note that $a < b$. If $a \in \sigma$, then $R$ and $g$ are compatible on $\<ab\>$ as $g \in \mathcal{F}$. If $a$ and $b$ are both in $\xi$, then $g(a) = g(b)$ and thus $R$ and $g$ are compatible on $\<ab\>$ as well.
\end{proof}

Hence, if $\xi \in [Y]^{<\omega}$ and $\sigma\xi$ is $R$-transitive then $n \not\in W^{\sigma\xi}_e$, by the choice of $g$. So, $(\sigma, Y)$ is a desired extension.
\end{proof}

Lemma \ref{lem:SEM-presv-D2} follows from Lemmata \ref{lem:EM-acceptable-longer} and \ref{lem:EM-acceptable-density}.

\subsection{Rainbow Ramsey Theorem}

From Proposition \ref{pro:random-presv} and \cite[Theorem 3.1]{Csima.Mileti:2009.rainbow}, we obtain a preservation property of $\RRT^2_2$.

\begin{corollary}\label{cor:RRT22-presv}
$\RRT^2_2$ admits preservation of the arithmetic hierarchy.
\end{corollary}

We shall see in Corollary \ref{cor:RRT-n-presv} that $\RRT^n_2$ in general does not have the preservation property in the above corollary. But some weaker preservation holds for $\RRT^3_2$.

\begin{theorem}\label{thm:RRT32-presv-D3}
$\RRT^3_2$ admits preservation of $\Delta^0_3$-definitions.
\end{theorem}

To prove the above theorem, we need the following lemmata, which are proved based on the forcing construction in \cite[\S 3.2 and \S4.5]{Wang:2014}. For the new notions appeared below, we refer the reader to \cite[\S 3.2 and \S4.5]{Wang:2014}. The first lemma is an analogous of \cite[Theorem 3.4]{Wang:2014}.

\begin{lemma}\label{lem:jump-cohesiv-presv-D3}
Every $Z'$-computable $\vec{A} = (A_n: n < \omega)$ admits a cohesive set preserving $\Delta^0_3$-definitions relative to $Z$.
\end{lemma}

\begin{proof}
It suffices to show that for every fixed $\vec{B} = (B_n: n < \omega)$ with no $B_n$ being $\Sigma^Z_2$ there exists an $\vec{A}$-cohesive $G$ such that no $B_n$ is $\Sigma^{Z \oplus G}_2$ either. For simplification, we assume that $Z = \emptyset$ and the reader can find that the proof below is relativizable.

We build an infinite binary tree $T$ and a map $f: T \to [\omega]^{<\omega}$ such that
\begin{itemize}
    \item $f$ is strictly increasing, i.e., if $\mu$ is a proper initial segment of $\nu \in T$ then $f(\mu)$ is a proper initial segment of $f(\nu)$, and
    \item if $P \in [T]$ and $G = \bigcup_m f(P \uh m)$ then $G$ is $\vec{A}$-cohesive and $B_n \not\in \Sigma^{G}_2$ for every $n$.
\end{itemize}
Then we can take $G = f(P) = \bigcup_m f(P \uh m)$ for any $P \in [T]$. We use multiple Mathias conditions of the form $((\sigma_\mu: \mu \in I), X)$ where $X$ is low (so non-$\Sigma^0_2$ sets are exactly non-$\Sigma^{X}_2$ sets), and build a sequence $(p_n: n < \omega)$ of decreasing conditions such that
\begin{itemize}
    \item each $p_n = ((\sigma_{n,\mu}: \mu \in I_n), X_n)$ and $I_0 = \{\emptyset\}$, $\sigma_{0,\emptyset} = \emptyset$ and $X_0 = \omega$,
    \item $I_n \neq I_{n+1}$ and if $\mu \in I_{n}$ then either $\mu \in I_{n+1}$ or both $\mu\<0\>$ and $\mu\<1\>$ are in $I_{n+1}$, and
    \item if $\mu \in I_{n+1} - I_n$ then $\sigma_{n,\mu^-}$ is a proper initial segment of $\sigma_{n+1,\mu}$ where $\mu^- = \mu \uh (|\mu| - 1)$.
\end{itemize}
So $T = \bigcup_n I_n$ is an infinite binary tree and $f: \mu \mapsto \sigma_{n,\mu}$ where $n = |\mu|$ is a strictly increasing map from $T$ to $[\omega]^{<\omega}$.

Below we build the desired $(p_n: n < \omega)$. By s-m-n theorem, we fix a computable function $s$ such that for any $C$ a number $n$ is in the $e$-th $\Sigma^C_2$ set if and only if $\Phi_{s(e,n)}(C)$ is partial where $\Phi_i$ is the $i$-th oracle Turing machine. For each $\mu \in 2^{<\omega}$, let $A_\mu = \bigcap_{\mu(i) = 1} A_i \cap \bigcap_{\mu(i) = 0} (\omega - A_i)$.

\begin{claim}\label{clm:jump-cohesiv-presv-D3}
Suppose that $p = ((\sigma_\mu: \mu \in I), X)$ is $(\vec{e}_\mu: \mu \in I)$-large and $X$ is low. Then for every $n$ and $e$ there exist $q = ((\tau_\nu: \nu \in J), Y)$ and $(\vec{d}_\nu: \nu \in J)$ such that
\begin{enumerate}
    \item $q$ is a $(\vec{d}_\nu: \nu \in J)$-large extension of $p$ and $Y$ is low;
    \item $J \neq I$ and if $\mu \in I$ then either $\mu \in J$ or both $\mu\<0\>$ and $\mu\<1\>$ are in $J$;
    \item if $\nu \in J$ extends $\mu \in I$ then $\tau_\nu - \sigma_\mu \in X \cap A_\mu$;
    \item if $\nu \in I \cap J$ then $\vec{d}_\nu = \vec{e}_\nu$ and for every $G \in (\sigma_\nu, X \cap A_\mu)$ there exists $i \in \vec{e}_\nu$ with $\Phi_i(G)$ partial;
    \item if $\nu \in J - I$ then $\dom \Phi_i(\tau_\nu) > \dom \Phi_i(\sigma_{\nu^-})$ for all $i \in \vec{e}_{\nu^-}$, and for some $m$ and $j = s(e,m)$ either $m \in B_n$ and $\vec{d}_\nu = \vec{e}_{\nu^-}\<j\>$ or $m \not\in B_n$ and $\vec{d}_{\nu} = \vec{e}_{\nu^-}$ and $\Phi_j(G)$ is partial for each $G \in (\tau_\nu, Y)$.
\end{enumerate}
\end{claim}

Here we assume that $\dom \Phi_i(\sigma) = \{0, 1, \ldots, m - 1\} = m$ for $\sigma \in [\omega]^{<\omega}$. From the first clause of (5), if $p$ is $(\vec{e}_\mu: \mu \in I)$-large and $P \in [T]$ extends $\mu \in I$ then we intend to have $\Phi_i(f(P))$ total for every $e \in \vec{e}_\mu$, while by the last part of (5) if $p$ is extended to $q$ and $j = s(e,m) \not\in \vec{d}_{\nu}$ then $p$ forces $\Phi_j(f(P))$ being partial for $P$ extending $\nu$. So (5) ensures that each $B_n$ is not $\Sigma^0_2$ relative to $f(P)$ for $P \in [T]$. That $T$ is infinite is guaranteed by (2) and that $f(P)$ is $\vec{A}$-cohesive for each $P \in [T]$ by (3).

\begin{proof}[Proof of Claim \ref{clm:jump-cohesiv-presv-D3}]
For $\mu \in I$, let $\rho_\mu \in [X \cap A_\mu]^{<\omega}$ be such that
$$
    (\forall i \in \vec{e}_\mu) \dom \Phi_i(\sigma_\mu \rho_\mu) > \dom \Phi_i(\sigma_\mu).
$$
Let $S$ be the set of $\mu \in I$ with $\rho_\mu$ defined, and let $J$ be the set of $\nu \in 2^{<\omega}$ such that either $\nu \in I - S$ or $\nu^- \in S$. Recall that as $p$ is $(\vec{e}_\mu: \mu \in I)$-large we intend to force the totality of $\Phi_i(G)$ for every $i \in \vec{e}_\mu$ and sufficiently generic $G \in (\sigma_\mu, X \cap A_\mu)$. But if $\mu \in I - S$ then our intention fails for $(\sigma_\mu, X \cap A_\mu)$. Nevertheless, by \cite[Lemma 3.8(1)]{Wang:2014}, $S \neq \emptyset$ and (2) of the claim holds.

We build a finite descending sequence of extensions $q_k$ of $p$ and define $q$ as the least of this sequence. Firstly we define $q_0$. For each $\nu \in J$, if $\nu \in I$ then let $\tau_{0,\nu} = \sigma_\nu$, otherwise let $\tau_{0,\nu} = \sigma_{\nu^-} \rho_{\nu^-}$. Pick $Y_0 \subseteq X$ such that $Y_0 =^* X$ and $\min Y_0 > \max \tau_{0,\nu}$ for all $\nu \in J$. Then $q_0 = ((\tau_{0,\nu}: \nu \in J), Y_0) \leq^*_M p$. Let $(\vec{d}_{0,\nu}: \nu \in J)$ be such that $\vec{d}_{0,\nu} = \vec{e}_\nu$ if $\nu \in I$ and $\vec{d}_{0,\nu} = \vec{e}_{\nu^-}$ if $\nu \in J - I$. Then $q_0$ is $(\vec{d}_{0,\nu}: \nu \in J)$-large, by \cite[Lemma 3.8(3)]{Wang:2014}.

Let $(\nu_k: k < |J - I|)$ enumerate $J - I$. For $k < |J - I|$, suppose that $q_k = ((\tau_{k,\nu}: \nu \in J), Y_k)$ is a $(\vec{d}_{k,\nu}: \nu \in J)$-large extension of $p$ with $Y_k$ low. For each $m$, append $s(e,m)$ to $\vec{d}_{k, \nu_k}$ in $(\vec{d}_{k,\nu}: \nu \in J)$ and denote the resulted sequence by $(\vec{c}_{k,m,\nu}: \nu \in J)$. Let
$$
    W = \{m: q_k \text{ is small for } (\vec{c}_{k,m,\nu}: \nu \in J)\}.
$$
By \cite[Definition 3.7]{Wang:2014}, $W$ is $\Sigma^0_2$ in $Y_k$ and thus $\Sigma^0_2$ as $Y_k$ is low. As $B_n \not\in \Sigma^0_2$, we can fix $m \in B_n \bigtriangleup W$. We define $q_{k+1}$ and $(\vec{d}_{k+1,\nu}: \nu \in J)$ by case.

\medskip

\emph{Case 1.} $m \in W$. By Low Basis Theorem, let $(X_i, \vec{\xi}_{i,\nu}: i < l, \nu \in J)$ be an $\vec{A}$-branching of $q_k$ with each $X_i$ low. By the $(\vec{d}_{k,\nu}: \nu \in J)$-largeness of $q_k$ and \cite[Lemma 3.8(2)]{Wang:2014}, pick $i < l$ and $\xi \in \vec{\xi}_{i,\nu_k}$ such that $((\tau_{k,\nu}: \nu \in J), X_i)$ is $(\vec{d}_{k,\nu}: \nu \in J)$-large and
$$
    (\exists y)(\forall \rho \in [X_i]^{<\omega}) \Phi_{s(e,m)}(\tau_{k,\nu_k}\xi\rho; y) \uparrow.
$$
Let $(\vec{d}_{k+1,\nu}: \nu \in J) = (\vec{d}_{k,\nu}: \nu \in J)$. Replace $\tau_{k,\nu_k}$ and $Y_k$ in $q_k$ by $\tau_{k,\nu_k} \xi$ and $X_i$ respectively and denote the resulted condition by $q_{k+1}$. Then $q_{k+1}$ is $(\vec{d}_{k+1,\nu}: \nu \in J)$-large by \cite[Lemma 3.8(3)]{Wang:2014}.

\medskip

\emph{Case 2.} $m \in B_n$. Let $q_{k+1} = q_k$ and $(\vec{d}_{k+1,\nu}: \nu \in J) = (\vec{c}_{k,m,\nu}: \nu \in J)$. As $m \not\in W$, $q_{k+1}$ is $(\vec{d}_{k+1,\nu}: \nu \in J)$-large.

\medskip

Finally, let $q = q_k$ and $(\vec{d}_{\nu}: \nu \in J) = (\vec{d}_{k,\nu}: \nu \in J)$ for $k = |J - I|$. By the construction above, $q$ and $(\vec{d}_{\nu}: \nu \in J)$ are as desired.
\end{proof}

Recall that $I_0 = \{\emptyset\}$. Let $\vec{e}_{0,\emptyset} = \emptyset$. For convenient, all conditions are considered $(\vec{e}_{0,\mu}: \mu \in I_0)$-large. With the above claim, we can define $(p_n: n < \omega)$ and $(\vec{e}_{n,\mu}: \mu \in I_n)$ such that
\begin{itemize}
    \item $p_0$ is already defined before Claim \ref{clm:jump-cohesiv-presv-D3};
    \item $p_{n+1} \leq^*_M p_n$ and $p_n$ is $(\vec{e}_{n,\mu}: \mu \in I_n)$-large;
    \item $I_n \not\subseteq I_{n+1}$ and if $\mu \in I_{n} - I_{n+1}$ then both $\mu\<0\>$ and $\mu\<1\>$ are in $I_{n+1}$;
    \item If $\nu \in I_{n+1}$ extends $\mu \in I_n$ then $\sigma_{n+1,\nu} - \sigma_{n,\mu} \in A_\mu$ and $\dom \Phi_i(\sigma_{n+1,\nu}) > \dom \Phi_i(\sigma_{n,\mu})$ for all $i \in \vec{e}_{n,\mu}$;
    \item For each $i$ and $e$, there exist $m$ and $n$ such that for every $\mu \in I_{n+1} - I_n$ exactly one of the followings holds:
    \begin{enumerate}
	\item[(a)] $m \not\in B_i$, $s(e,m) \not\in \vec{e}_{n+1,\mu}$ and $\Phi_{s(e,m)}(G)$ is partial for all $G \in (\sigma_{n+1,\mu}, X_{n+1})$,
	\item[(b)] $m \in B_i$ and $s(e,m) \in \vec{e}_{n+1,\mu}$.
    \end{enumerate}
\end{itemize}
So $T = \bigcup_n I_n$ is an infinite binary tree and $f: \mu \mapsto \sigma_{n, \mu}$ where $n = |\mu|$ maps $T$ to $[\omega]^{<\omega}$ in a strictly increasing manner. For each $P \in [T]$, $f(P)$ is $\vec{A}$-cohesive. For $e$ and $i$, let $m$ and $n$ witness the last bullet point above and let $U$ be the $e$-th $\Sigma^G_2$ set where $G = f(P)$. If (a) holds then $m \in U - B_i$; otherwise (b) holds and $m \in B_i - U$ by (5) of Claim \ref{clm:jump-cohesiv-presv-D3}. So no $B_i$ is $\Sigma^G_2$.
\end{proof}

The next lemma is an analogous of \cite[Lemma 4.11]{Wang:2014}.

\begin{lemma}\label{lem:RRT22-presv-D3}
If $f: [\omega]^2 \to \omega$ is $2$-bounded, stable and $Z'$-computable then there exists an infinite $f$-rainbow preserving $\Delta^0_3$-definitions relative to $Z$.
\end{lemma}

\begin{proof}
We present a relativizable proof for $Z = \emptyset$. By Proposition \ref{pro:random-presv}, sufficiently random sets preserve the arithmetic hierarchy, and by \cite[Lemma 4.3]{Wang:2014} each sufficiently random set computes an infinite $Y$ such that $f(u,v) \neq f(x,y)$ for $(u,v), (x,y) \in [Y]^2$ with distinct $v$ and $y$. So by replacing $\omega$ with some $Y$ as above if necessary, we may assume that $f(u,v) \neq f(x,y)$ for all $(u,v)$ and $(x,y)$ with distinct $v$ and $y$. Moreover, as in \cite[\S 4.1]{Wang:2014}, we can assume that $f(x,y) = \<w,y\>$ where $w = \min\{v: f(v,y) = f(x,y)\}$.

We construct an infinite $f$-rainbow $G$ preserving $\Delta^0_3$-definitions by the following complexity analysis of the forcing argument in \cite[\S 4.5]{Wang:2014}:
\begin{enumerate}
    \item We work with large conditions $(\sigma,X,\vec{h})$ which are low (i.e., $X \oplus \vec{h}$ is low);
    \item Observe that it is a $\Sigma^0_2$ question whether a large condition $p = (\sigma,X,\vec{h})$ passes an $e$-test at $y$ (defined after the proof of \cite[Lemma 4.16]{Wang:2014}), if $p$ is low;
    \item By s-m-n theorem, we fix a computable function $s$ such that for any $C$ a number $n$ is in the $e$-th $\Sigma^C_2$ set if and only if $\Phi_{s(e,n)}(C)$ is partial where $\Phi_i$ is the $i$-th oracle Turing machine;
    \item So for a given condition $p$ and an index $e$, the following set is $\Sigma^0_2$ in $X \oplus \vec{h}$ and thus $\Sigma^0_2$ as $X \oplus \vec{h}$ is low:
    $$
	W = \{y: (\exists x) (p \text{ passes the $s(e,y)$-test at } x)\};
    $$
    \item If $A \not\in \Sigma^0_2$ then we can fix $y \in A \bigtriangleup W$;
    \item If $y \in A - W$ then $p$ fails the $s(y)$-test at every $x$ and by \cite[Lemma 4.18]{Wang:2014} we have $\Phi_{s(e,y)}(G)$ total for $G$ being sufficiently generic with respect to large conditions, thus $y$ is not in the $e$-th $\Sigma^G_2$ set;
    \item If $y \in W - A$ then $p$ passes the $s(e,y)$-test at some $x$ and by \cite[Lemma 4.17]{Wang:2014} we can extend $p$ to $q$ which forces $\Phi_{s(e,y)}(G; x) \uparrow$ for $G$ sufficiently generic with respect to large conditions and thus $y$ is in the $e$-th $\Sigma^G_2$ set;
    \item Hence for a list of properly $\Delta^0_3$ sets $(A_i: i < \omega)$ by forcing with large conditions we can obtain an infinite $f$-rainbow $G$ so that $A_i \not\in \Sigma^G_2$ for every $i$. So $G$ preserves $\Delta^0_3$-definitions.
\end{enumerate}
\end{proof}

\begin{proof}[Proof of Theorem \ref{thm:RRT32-presv-D3}]
Recall that we may identify $\sigma \in [\omega]^{<\omega}$ with $\ulcorner \sigma \urcorner$ where $\ulcorner \cdot \urcorner$ is a fixed computable bijection mapping $[\omega]^{<\omega}$ onto $\omega$.

Fix $Z$ and a $Z$-computable $2$-bounded coloring $f: [\omega]^3 \to \omega$. By passing to an infinite $f$-computable subset if necessary, we assume that $f(\sigma\<x\>) \neq f(\tau\<y\>)$ for all $\sigma\<x\>, \tau\<y\> \in [\omega]^{3}$ with distinct $x$ and $y$. As in the proof of Lemma \ref{lem:jump-cohesiv-presv-D3}, we may assume that
$$
    f(\sigma\<x\>) = \min \{\tau\<x\>: f(\tau\<x\>) = f(\sigma\<x\>)\}
$$
for all $\sigma\<x\> \in [\omega]^3$. Let $\vec{A} = (A_{\sigma,\tau}: \sigma, \tau \in [\omega]^{2})$ be such that
$$
    A_{\sigma,\tau} = \{x: f(\sigma\<x\>) = \tau\<x\>\}
$$
for each $\sigma, \tau \in [\omega]^{2}$. Then $\vec{A} \leq_T f$. By Corollary \ref{cor:RT22-presv} below and that $\RCA + \RT^2_2 \vdash \COH$, let $C$ be an $\vec{A}$-cohesive set preserving $\Delta^0_3$-definitions relative to $Z$.

Let $g: [C]^{2} \to \omega$ be such that
$$
    g(\sigma) = \lim_{x \in C} f(\sigma\<x\>)
$$
for all $\sigma \in [C]^{2}$. By the cohesiveness of $C$, $g$ is a well-defined total function. We may assume that $g(x,y) = \min\{\<i,j\>: g(i,j) = g(x,y)\}$. Moreover $g \leq_T (f \oplus C)'$. As $f$ is $2$-bounded, $g$ is $2$-bounded as well. Let $\vec{B} = (B_{n,x}: n < \omega, x \in C)$ be such that
$$
    B_{n,x} = \{y \in C: g(x,y) = n\}.
$$
Then $\vec{B} \leq_T (f \oplus C)'$. By Lemma \ref{lem:jump-cohesiv-presv-D3}, let $D \in [C]^{\omega}$ be $\vec{B}$-cohesive and preserving $\Delta^0_3$-definitions relative to $Z \oplus C$. As $C$ preserves $\Delta^0_3$-definitions relative to $Z$, so does $C \oplus D$.

Let $h$ be $g$ restricted to $[D]^2$. Then $h \leq_T (f \oplus C \oplus D)'$ and $h$ is $2$-bounded and stable. Apply Lemma \ref{lem:RRT22-presv-D3}, we get an infinite $h$-rainbow $H$ which is a subset of $D$ and preserves $\Delta^0_3$-definitions relative to $Z \oplus C \oplus D$. By the definition of $h$, $H$ is also a $g$-rainbow. By the definition of $g$, $Z \oplus C \oplus D \oplus H$ computes an infinite $f$-rainbow $G \subseteq H$ (see the proof of \cite[Lemma 3.2]{Wang:2013}). Thus
$$
    \Delta^Z_3 - \Sigma^Z_2 \subseteq \Delta^Z_3 - \Sigma^{Z \oplus C \oplus D}_2 \subseteq \Delta^Z_3 - \Sigma^{Z \oplus C \oplus D \oplus H}_2 \subseteq \Delta^Z_3 - \Sigma^{Z \oplus G}_2.
$$
So $G$ preserves $\Delta^0_3$-definitions relative to $Z$.
\end{proof}

\subsection{More preservations}

We have seen some preservation results that need substantial proofs. Here we list a few that follow easily from the above preservation results. Firstly, let us recall some additional notation and consequences of Ramsey's Theorem:
\begin{itemize}
  \item We denote $(\forall c < \infty) \RT^n_c$ by $\RT^n$ and $(\forall n < \infty) \RT^n$ by $\RT$.
  \item An apparently weaker consequence of $\RT^n$ is the so-called \emph{Achromatic Ramsey Theorem} ($\ART^n_{<\infty, d}$): for every finite coloring $f$ of $[\omega]^n$ there exists an infinite subset $H$ such that $f([H]^n)$ contains at most $d$ many colors.
  \item For a coloring $f: [\omega]^n \to \omega$, a \emph{free set} is a set $H$ such that $f(\sigma) \not\in H - \sigma$ for all $\sigma \in [H]^n$, and a set $G$ is \emph{thin} if $f([G]^n) \neq \omega$. The \emph{Free Set Theorem} ($\FS$) asserts that every $f: [\omega]^n \to \omega$ for finite $n$ admits an infinite free set, and the \emph{Thin Set Theorem} ($\TS$) asserts that every $f: [\omega]^n \to \omega$ for finite $n$ admits an infinite thin set. Over $\RCA$, $\RT$ implies $\FS$ and $\FS$ implies $\TS$ (\cite{Cholak.Giusto.ea:2005.freeset}).
\end{itemize}

\begin{corollary}\label{cor:presv-P1}
The following statements admit preservation of $\Sigma^0_1$- and $\Pi^0_1$-definitions: $\RT^2$ and its consequences over $\RCA$, $\FS$ and its consequences (e.g., $\TS$, $\RRT$) over $\RCA$, and for each $n$ almost all instances of $\ART^n_{<\infty, d}$.
\end{corollary}

\begin{proof}
By Seetapun \cite{Seetapun.Slaman:1995.Ramsey} and Propositions \ref{pro:sim-avoidance-presv-P1} and \ref{pro:sep-by-presv}, $\RT^2$ and its consequences over $\RCA$ admit preservation of $\Sigma^0_1$- and $\Pi^0_1$-definitions. The other preservations follow from Propositions \ref{pro:sim-avoidance-presv-P1} and \ref{pro:sep-by-presv} and the author's work \cite{Wang:2014.ramseyan}.
\end{proof}

By Jockusch \cite{Jockusch:1972.Ramsey}, there exists a computable $f: [\omega]^3 \to 2$ such that every infinite $f$-homogeneous set computes the halting problem. Hence $\RT^n_2$ does not admit preservation of $\Sigma^0_1$-definitions for any $n > 2$.

By an application of Theorem \ref{thm:WKL-presv}, we obtain a stronger result for $\RT^2$.

\begin{corollary}\label{cor:RT22-presv}
If $\Phi$ is a $\Pi^1_2$ consequence of $\RCA + \RT^2$ then $\Phi$ admits preservation of $\Xi$-definitions simultaneously for all $\Xi$ in $\{\Sigma^0_{n+1}, \Pi^0_{n+1}, \Delta^0_{n+2}: n > 0\}$.
\end{corollary}

\begin{proof}
By Proposition \ref{pro:sep-by-presv}, it suffices to prove the corollary for $\Phi$ being $\RT^2$. Let $Z$ be fixed and $f$ be a finite coloring of pairs computable in $Z$. By Theorem \ref{thm:WKL-presv}, let $P$ be such that $P$ is PA over $Z'$ and $P$ preserves the arithmetic hierarchy relative to $Z'$. By relativizing the argument in \cite[\S 4]{Cholak.Jockusch.ea:2001.Ramsey}, we can obtain an infinite $f$-homogeneous set $H$ with $(Z \oplus H)' \leq_T P$.

Fix $\Xi$ in $\{\Sigma^0_{n+1}, \Pi^0_{n+1}, \Delta^0_{n+2}: n > 0\}$ and a properly $\Xi^Z$ set $A$. If $n > 0$ and $\Xi$ is $\Sigma^Z_{n+1}$ then $A$ is properly $\Sigma^{Z'}_n$. By the choice of $P$, $A$ is properly $\Sigma^P_n$. As $Z' \leq_T (Z \oplus H)' \leq_T P$, $A$ is properly $\Sigma^{(Z \oplus H)'}_{n}$ and thus properly $\Sigma^{Z \oplus H}_{n+1}$. The remaining cases of $\Xi$ can be proven similarly.
\end{proof}

However, we shall see in the next section that $\RT^2_2$ does not admit preservation of $\Delta^0_2$-definitions.

\section{Non-preservations}\label{s:n-presv}

In the last section, we have learned some examples in Ramsey theory and computability that are weak in terms of definability strength. Here we present two $\Pi^1_2$ propositions in Ramsey theory that are relatively strong: one is the so-called Ascending or Descending Sequence principle and one the Thin Set Theorem.

\subsection{Ascending or Descending Sequence}

The Ascending or Descending Sequence principle ($\ADS$) asserts that every infinite linear order has a infinite ascending or descending suborder. A linear order of type a suborder of $\omega + \omega^*$ is called a \emph{stable linear order}, where $\omega^*$ is the reverse order of $\omega$. The Stable Ascending or Descending Sequence principle ($\SADS$) is $\ADS$ restricted to stable linear orders. Hirschfeldt and Shore \cite{Hirschfeldt.Shore:2007} proved that $\SADS$ is strictly weaker than $\ADS$ and $\ADS$ is strictly weaker than $\RT^2_2$ (over $\RCA$).

The following theorem is essentially an observation of Jockusch \cite[Corollary 2.14]{Hirschfeldt.Shore:2007}.

\begin{theorem}\label{thm:SADS-n-presv}
$\SADS$ does not admit preservation of $\Delta^0_2$-definitions.
\end{theorem}

\begin{proof}
By Harizanov \cite{Harizanov:1998}, we can take a computable stable linear order $<_L$ such that both the $\omega$-part $U = \{i: (\forall^\infty j) (i <_L j)\}$ and the $\omega^*$-part $\omega - U$ of $<_L$ are properly $\Delta^0_2$. If $S$ is an infinite $<_L$-ascending sequence then $U$ is $\Sigma^0_1$ in $S$, since $U = \{i: (\exists j \in S) (i <_L j)\}$. Similarly, if $S$ is an infinite $<_L$-descending sequence then $\omega - U$ is $\Sigma^0_1$ in $S$. So $\SADS$ does not admit preservation of $\Delta^0_2$-definitions.
\end{proof}

So we obtain a non-preservation property of $\ADS$ and $\RT^2_2$ by Proposition \ref{pro:sep-by-presv} and the above theorem.

\begin{corollary}\label{cor:RT22-n-presv-D2}
Neither $\ADS$ nor $\RT^2_2$ admits preservation of $\Delta^0_2$-definitions.
\end{corollary}

\subsection{Thin Set Theorem}

In this subsection we present a non-preservation theorem of $\TS$. Let $\TS^n$ denote the instance of $\TS$ for colorings of $[\omega]^n$ and let $\STS^n$ denote $\TS^n$ for stable colorings. Clearly, $\RCA + \TS^n \vdash \STS^n$.

\begin{theorem}\label{thm:TS-n-presv}
For each $n > 1$, there exists a computable stable function $f: [\omega]^n \to \omega$ such that every positive $i < n$ and every infinite $f$-thin set $X$ correspond to some $B \in (\Delta^0_{i+1} - \Sigma^0_i) \cap \Sigma^X_i$. Hence neither $\STS^n$ nor $\TS^n$ admits $\Delta^0_{i+1}$-preservation for any positive $i < n$.
\end{theorem}

The second part of the above theorem follows from the first part and Proposition \ref{pro:presv-equiv}. We prove the first part of Theorem \ref{thm:TS-n-presv} by induction on $n$ and the following technical lemma.

\begin{lemma}\label{lem:TS-n-presv}
Let $n > 0$ and $\bar{f}: [\omega]^n \to \omega$ be such that
\begin{itemize}
    \item $\bar{f}$ is computable in $\emptyset'$, and
    \item if $n > 1$ then $\lim_x \bar{f}(\sigma,x)$ exists for all $\sigma \in [\omega]^{n-1}$.
\end{itemize}
Then there exist $f: [\omega]^{n+1} \to \omega$ and a sequence of sets $(B_i: i < \omega)$ such that
\begin{itemize}
    \item $f$ is computable and $\lim_y f(\xi,y)$ exists for all $\xi \in [\omega]^n$,
    \item if $n > 1$ and $\sigma \in [\omega]^{n-1}$ then $\lim_x \lim_y f(\sigma,x,y)$ exists and equals $\lim_x \bar{f}(\sigma,x)$,
    \item $B_i$ is $\Delta^0_2$ but not $\Sigma^0_1$, and
    \item if $X$ is infinite and $f$-thin then $B_i \in \Sigma^X_1$ for some $i$.
\end{itemize}
\end{lemma}

\begin{proof}[Proof of Theorem \ref{thm:TS-n-presv}]
By induction on $n > 1$, we construct a function $f$ which satisfies Theorem \ref{thm:TS-n-presv} and has an additional property that $\lim_x f(\sigma,x)$ exists for all $\sigma \in [\omega]^{n-1}$.

For $n = 2$, let $\bar{f}$ be any $\emptyset'$-computable function on $\omega$. Fix $f$ be as in Lemma \ref{lem:TS-n-presv} for $\bar{f}$. Then $f$ is as desired.

For $n > 2$, by relativizating the induction hypothesis we fix a $\emptyset'$-computable $\bar{f}: [\omega]^{n-1} \to \omega$ such that
\begin{itemize}
    \item $\lim_x \bar{f}(\sigma,x)$ exists for all $\sigma \in [\omega]^{n-2}$,
    \item if $i < n-1$ is positive and $X$ is an infinite $\bar{f}$-thin set then there exists $B \in (\Delta^{\emptyset'}_{i+1} - \Sigma^{\emptyset'}_i) \cap \Sigma^{\emptyset' \oplus X}_i$.
\end{itemize}
Apply Lemma \ref{lem:TS-n-presv} to get $f$ and $(B_i: i < \omega)$ corresponding to $\bar{f}$ and $n-1$. Suppose that $i < n$ is positive and $X$ is an infinite $f$-thin set. If $i = 1$ then $B_k \in \Sigma^X_1$ for some $k$ and thus $B_k \in (\Delta^{0}_{i+1} - \Sigma^{0}_i) \cap \Sigma^{X}_i$. Assume that $i > 1$ and $k \not\in f([X]^n)$. Then
$$
    \lim_y f(\xi,y) = \lim_{y \in X} f(\xi,y) \neq k
$$
for all $\xi \in [X]^{n-1}$ and thus
$$
    \lim_x \bar{f}(\sigma,x) = \lim_x \lim_y f(\sigma,x,y) = \lim_{x \in X} \lim_{y \in X} f(\sigma,x,y) \neq k
$$
for all $\sigma \in [X]^{n-2}$. As $\bar{f}$ is $\emptyset'$-computable, we can find $Y \in [X]^\omega$ such that $Y$ is computable in $\emptyset' \oplus X$ and $k \not\in \bar{f}([Y]^{n-1})$. By the choice of $\bar{f}$, there exists $B \in (\Delta^{\emptyset'}_{i} - \Sigma^{\emptyset'}_{i-1}) \cap \Sigma^{\emptyset' \oplus Y}_{i-1}$. So
$$
    B \in (\Delta^{0}_{i+1} - \Sigma^{0}_{i}) \cap \Sigma^{\emptyset' \oplus X}_{i-1} \subseteq (\Delta^{0}_{i+1} - \Sigma^{0}_{i}) \cap \Sigma^{X'}_{i-1} \subseteq (\Delta^{0}_{i+1} - \Sigma^{0}_{i}) \cap \Sigma^{X}_{i}.
$$
Hence $f$ is as desired.
\end{proof}

To prove Lemma \ref{lem:TS-n-presv}, we build several objects:
\begin{enumerate}
 \item A computable function $f: [\omega]^{n+1} \to \omega$ as required by the lemma. Let $A_i = \{\xi: \lim_s f(\xi,s) = i\}$.
 \item A computable trinary function $g$ which approximates the sequence $(B_i: i < \omega)$ in the following way: $\lim_s g(i,x,s)$ exists for each $(i,x)$ and \emph{no} $B_i = \{x: \lim_s g(i,x,s) = 1\}$ is $\Sigma^0_1$.
 \item A uniformly $\Sigma^0_1$ sequence $(U_i: i < \omega)$ such that
 \begin{itemize}
  \item if $x \in B_i$ then $\<x\>\xi \in U_i$ for all $\xi \not\in A_i$ with $\min \xi$ sufficiently large, and
  \item if $x \not\in B_i$ and $\<x\>\xi \in U_i$ then $\xi \in A_i$.
 \end{itemize}
\end{enumerate}

We guarantee that $\lim_s f(\xi,s)$ exists by ensuring that $f(\xi,s)$ changes at most finitely often. We apply the same strategy to achieve the existence of $\lim_s g(i,x,s)$. To achieve $\lim_x \lim_y f(\sigma,x,y) = \lim_x \bar{f}(\sigma,x)$ for $n > 1$ and $\sigma \in [\omega]^{n-1}$, we fix a computable approximation $(\bar{f}_s: s < \omega)$ of $\bar{f}$ and ensure that $f(\sigma,x,s) = \bar{f}_s(\sigma,x)$ for $x$ and $s$ sufficiently large.

To make $B_i \not\in \Sigma^0_1$, we employ a finite injury argument to satisfy the following requirements:
$$
    R_{i,j}: B_i \neq W_j.
$$
We follow the Friedberg-Muchnik construction to meet a single $R_{i,j}$: we pick a \emph{witness} for $R_{i,j}$, say $b$, and put it in $B_i$ by defining $g(i,b,s) = 1$. If $b \in W_{j,t}$ at stage $t > s$ then we remove $b$ from $B_i$ by defining $g(i,b,t) = 0$.

To meet the first part of condition (3) above, at each stage $s$ we enumerate $\<b\>\xi$ in $U_{i,s}$ if $b < \min \xi \leq \max \xi < s$, $g(i,b,s) = 1$ and $f(\xi,s) \neq i$. However, this action may cause some problem for the second part of condition (3), since later $b$ could be removed from $B_i$ for the sake of some $R_{i,j}$. To work around this problem, if we remove $b$ from $B_i$ at stage $t > s$ then we put $\xi$ in $A_i$ by defining $f(\xi,t') = i$ for all $t' \geq t$. Note that at each stage we only enumerate a finite part of $U_i$. So the above action defining $f(\xi,t)$ affects at most finitely many $\xi$. Thus, if $n > 1$ and $\sigma \in [\omega]^{n-1}$ then this action causes $\lim_y f(\sigma,x,y) \neq \bar{f}(\sigma,x)$ for only finitely many $x$. We initialize all $R_{i',j'}$ with lower priorities, so that eventually they will have witnesses greater than $\min \sigma$. Thus at later stages $R_{i',j'}$ with lower priority will not require $\lim_y f(\sigma,x',y) \neq \bar{f}(\sigma,x')$ for any $x'$.

\medskip

\noindent \emph{The construction.}

Recall that we fix a computable approximation $(\bar{f}_s: s < \omega)$ of $\bar{f}$.

At stage $s$, if $s > 0$ then we assume that:
\begin{itemize}
 \item $f$ is defined on $[s]^{n+1}$;
 \item $g(i,x,r)$ is defined for all $(i,x)$ and $r < s$ and if $g(i,x,s-1) = 1$ then both $i$ and $x$ are less than $s$;
 \item each $U_{i,s-1}$ is finite.
\end{itemize}

The construction at stage $s$ consists of three parts.

(i) We take care of $R_{i,j}$'s here. A requirement $R_{i,j}$ \emph{requires attention} if either of the following conditions holds:
\begin{itemize}
 \item $R_{i,j}$ does not have a witness defined;
 \item $R_{i,j}$ has a witness (say $b$) defined and $g(i,b,s-1) = W_{j,s}(b) = 1$.
\end{itemize}

Pick the least $\<i,j\>$ with $R_{i,j}$ requiring attention and say that $R_{i,j}$ \emph{receives attention}.

Suppose that $R_{i,j}$ does not have a witness. Perform the following actions:
\begin{itemize}
 \item let $s$ be its witness and let $g(i,s,s) = 1$;
 \item for $\<i',j'\> < \<i,j\>$, let $g(i',b,s) = g(i',b,s-1)$ where $b$ is the witness of $R_{i',j'}$;
 \item proceed to (ii).
\end{itemize}

Suppose that $R_{i,j}$ has a witness $b$ defined. Perform the following actions:
\begin{itemize}
 \item let $g(i,b,s) = 0$;
 \item for each $\<i',j'\> < \<i,j\>$ and the witness $b'$ of $R_{i',j'}$, let $g(i',b',s) = g(i',b',s-1)$;
 \item for each $\<i',j'\> > \<i,j\>$ and the witness $b'$ of $R_{i',j'}$, let $g(i',b',s) = 1$ and let the witness of $R_{i',j'}$ be undefined. In other words, $R_{i',j'}$ is initialized and it will not have a witness defined until it receives attention again.
\end{itemize}

(ii) We define $f$ and $g$.
\begin{itemize}
 \item If $s = 0$ then let $g(i,x,s) = 0$ for all $i$ and $x$; otherwise, let $g(i,x,s) = g(i,x,s-1)$ for each $(i,x)$ such that $g(i,x,s)$ is not defined in (i).
 \item If $\xi \in [s]^n$ and $\<b\>\xi \in U_{i,s-1}$ for an active witness $b$ of some $R_{i,j}$ and if $g(i,b,s) = 0$, then let $f(\xi,s) = i$; otherwise let $f(\xi,s) = \bar{f}_s(\xi)$.
\end{itemize}

(iii) We define $U_{i,s}$ by
$$
    U_{i,s} = U_{i,s-1} \cup \{\<x\>\xi: x < \min \xi \leq \max \xi < s, g(i,x,s) = 1, f(\xi,s) \neq i\}.
$$

This ends the construction at stage $s$.

\medskip

\noindent \emph{The verification.}

\begin{lemma}\label{lem:TS-R}
Each $R_{i,j}$ receives attention finitely often and is satisfied.
\end{lemma}

Note that in this lemma $B_i$ is to be understood as a $\Sigma^0_2$ set.

\begin{proof}
We prove by induction on $\<i,j\>$. Fix $\<i,j\>$ and assume that all $R_{i',j'}$'s with $\<i',j'\> < \<i,j\>$ stop receiving attention after stage $s_0$. We may assume that $R_{i,j}$ receives attention and has its witness $x = s_0$ defined at stage $s_0$. As no $R_{i',j'}$ with $\<i',j'\> < \<i,j\>$ receives attention after stage $s_0$, $R_{i,j}$ has its witness $b = s_0$ at all later stages.

By the construction, $g(i,b,s_0) = 1$. If $b \not\in W_j$ then $R_{i,j}$ receives no attention after stage $s_0$ and is satisfied since $B_i(b) = \lim_s g(i,b,s) = 1 \neq 0 = W_j(b)$. Suppose that $R_{i,j}$ receives attention again at stage $s_1 > s_0$. Then $g(i,b,s_1-1) = W_{j,s_1}(b) = 1$ and $g(i,b,s_1) = 0$. By the construction, $R_{i,j}$ receives no attention after stage $s_1$ and $\lim_s g(i,b,s) = 0 \neq 1 = W_j(b)$, thus $R_{i,j}$ is satisfied.
\end{proof}

\begin{lemma}\label{lem:TS-f}
$f$ is well-defined and computable, $\lim_s f(\xi,s)$ exists for all $\xi \in [\omega]^{n}$ and if $n > 1$ then $\lim_x \lim_s f(\sigma,x,s) = \lim_x \bar{f}(\sigma,x)$ for all $\sigma \in [\omega]^{n-1}$.
\end{lemma}

\begin{proof}
Suppose that $f(\xi,s) \neq \bar{f}_s(\xi)$. Then at stage $s$ there exists $R_{i,j}$ such that $R_{i,j}$ has an active witness $b$, $\<b\>\xi \in U_{i,s-1}$ and $g(i,b,s) = 0$. By (i) and (iii) of the construction, there exists $s_0 \leq s$ with $g(i,b,s_0-1) = 1 \neq 0 = g(i,b,s_0)$ and $\<b\>\xi \in U_{s_0-1}$. At stage $s_0$, $R_{i,j}$ receives attention and all $R_{i',j'}$'s with $\<i',j'\> > \<i,j\>$ are initialized. So at any stage $t \geq s_0$, $R_{i',j'}$ with $\<i',j'\>$ cannot have an active witness $b' < \min \xi$. It follows that at stage $s$ the above $R_{i,j}$ is unique. Hence $f$ is well-defined. The construction guarantees that $f$ is computable.

To prove the existence of $\lim_s f(\xi,s)$, pick a stage $s$ such that no $R_{i,j}$ with a witness less than $\min \xi$ receives attention after stage $s$. At every stage $t > s$, either there exists exactly one fixed $R_{i,j}$ with an active witness $b$ such that $\<b\>\xi \in U_{i,t}$, or there is no such $R_{i,j}$. In the former case $\lim_s f(\xi,s) = i$ and in the latter $\lim_s f(\xi,s) = \lim_s \bar{f}_s(\xi) = \bar{f}(\xi)$.

Suppose that $n > 1$, $\sigma \in [\omega]^{n-1}$ and $f(\sigma,x,s_1) \neq \bar{f}_{s_1}(\sigma,x)$. Then there is exactly one $R_{i,j}$ with an active witness $b$ such that $\<b\>\sigma\<x\> \in U_{i,s_1-1}$ and $g(i,b,s_1) = 0$. Let $s_0 \leq s_1$ be the stage such that $g(i,b,s_0 - 1) = 1$ and $g(i,b,s_0) = 0$ and $\<b\>\sigma\<x\> \in U_{i,s_0-1}$. Then $R_{i,j}$ receives attention and all $R_{i',j'}$ with $\<i',j'\> > \<i,j\>$ are initialized at stage $s_0$. So at any stage $t > s_0$, if $\<i',j'\> > \<i,j\>$ and $R_{i',j'}$ has an active witness $b'$ then $b' > \max \sigma$ and thus there exists \emph{no} $\<b'\>\sigma\<x'\> \in U_{i',t}$. It follows that $\lim_s f(\sigma,x,s) \neq \bar{f}(\sigma,x)$ for at most finitely many $x$. As $\lim_x \bar{f}(\sigma,x)$ exists, $\lim_x \lim_s f(\sigma,x,s) = \lim_x \bar{f}(\sigma,x)$.
\end{proof}

\begin{lemma}\label{lem:TS-g}
$g$ is computable and stable.
\end{lemma}

\begin{proof}
By the construction, $g(i,x,s-1) \neq g(i,x,s)$ happens only if at stage $s$ some $R_{i,j}$ with an active witness $b \leq x$ receives attention. But there are at most finitely many $R_{i,j}$'s having witnesses not greater than $x$. So we can pick a stage $t$ such that no $R_{i,j}$ with an active witness $b \leq x$ receives attention after stage $t$. Then $g(i,x,t) = \lim_s g(i,x,s)$.
\end{proof}

Recall that $A_i = \{\xi: \lim_s f(\xi,s) = i\}$.

\begin{lemma}\label{lem:TS-U}
For each $i$, if $b \in B_i$ then $\<b\>\xi \in U_i$ for all $\xi \not\in A_i$ with $\min \xi$ sufficiently large, and if $b \not\in B_i$ and $\<b\>\xi \in U_i$ then $\xi \in A_i$.
\end{lemma}

\begin{proof}
It follows from (iii) of the construction at each stage and the stability of $f$ and $g$ that if $b \in B_i$ then $\<b\>\xi \in U_i$ for all $\xi \not\in A_i$ with $\min \xi > b$.

Suppose that $b \not\in B_i$ and $\<b\>\xi \in U_i$. By the construction of $U_i$, there is a stage $s_0$ such that $g(i,b,s_0) = 1$ and $f(\xi,s_0) \neq i$. As $b \not\in B_i$, there is a stage $s_1 > s_0$ such that $g(i,b,s_1 - 1) = 1$ and $g(i,b,s_1) = 0$. By (i) of the construction at stage $s_1$, $b$ is the active witness of some $R_{i,j}$, $R_{i,j}$ receives attention at stage $s_1$ and $f(\xi,s_1) = i$.

If some $R_{i',j'}$ with $\<i',j'\> < \<i,j\>$ receives attention at a stage $s > s_1$ then $g(i,b,s) = 1$ and $b$ cannot become an active witness for any requirement at any stage $t > s$. By (ii) of the construction, $b \in B_i$, contradicting our choice of $b$. So $R_{i,j}$ is not initialized after stage $s_1$ and $g(i,b,s) = 0$ for any $s \geq s_1$. By (ii) of the construction, $f(\xi,s) = i$ at every stage $s > s_1$. Hence $\lim_s f(\xi,s) = i$ and $\xi \in A_i$.
\end{proof}

It follows from Lemma \ref{lem:TS-f} that $f$ is as desired. By Lemmata \ref{lem:TS-g} and \ref{lem:TS-R}, $B_i = \{x: \lim_s g(i,x,s) = 1\}$'s yield a uniformly $\Delta^0_2$ sequence $(B_i: i < \omega)$ with each member not in $\Sigma^0_1$. If $X$ is an infinite $f$-thin set then $i \not\in f([X]^{n+1})$ for some $i$ and thus $B_i \in \Sigma^X_1$ by Lemma \ref{lem:TS-U}. So we have proven Lemma \ref{lem:TS-n-presv}.

\begin{corollary}\label{cor:FS-n-presv}
If $n > 1$ then $\FS^n$ does not admit preservation of $\Delta^0_{i+1}$-definitions for any positive $i < n$.
\end{corollary}

\begin{proof}
By Cholak et al. \cite[Theorem 3.2]{Cholak.Giusto.ea:2005.freeset}, $\RCA + \FS^n \vdash \TS^n$. So the corollary follows from Proposition \ref{pro:sep-by-presv} and Theorem \ref{thm:TS-n-presv}.
\end{proof}

Recently in private communications, Patey showed that $\RCA + \RRT^{2n+1}_2 \vdash \STS^{n+1}$ for $n > 0$. So the definability strength of $\RRT^{2n+1}_2$ for $n > 0$ is strictly stronger than that of $\RRT^2_2$.

\begin{corollary}\label{cor:RRT-n-presv}
For $n > 0$, $\RRT^{2n+1}_2$ does not admit preservation of $\Delta^0_{i+1}$-definitions for any positive $i \leq n$.
\end{corollary}

\section{Conclusion}\label{s:conclusion}

We summarize the known preservations and non-preservations in Table \ref{table} with references in parenthese. For simplification, we omit some easy consequences of the results in Table \ref{table}. For example, it is omitted that $\RT^2_2$ does not admit preservation of $\Delta^0_2$-definitions (Corollary \ref{cor:RT22-n-presv-D2}).

\renewcommand{\arraystretch}{1.25}
\begin{table}
  \centering
    \begin{tabu}{l|l|l}
      \hline
      & Preservations & Non-preservations \\
      \hline
      Arithmetic Hierarchy & Cohen generics (\ref{pro:generic-presv}) & \\
	& random reals (\ref{pro:random-presv}) & \\
	& $\Pi^0_1 \operatorname{G}$ (\ref{cor:PiG-presv}) & \\
	& $\WKL$ (\ref{thm:WKL-presv}) & \\
	& $\RRT^2_2$ (\ref{cor:RRT22-presv}) & \\
      \hline
      $\Sigma^0_{i+1}, \Pi^0_{i+1}, \Delta^0_{i+2}$ ($i > 0$) & $\RT^2$ (\ref{cor:RT22-presv}) & \\
      \hline
      $\Delta^0_{i+1}$ ($0 < i \leq n, n > 0$) & & $\STS^{n+1}$ (\ref{thm:TS-n-presv}) \\
      & & $\RRT^{2n+1}_2$ ($n > 0$, \ref{cor:RRT-n-presv}) \\
      \hline
      $\Delta^0_3$ & $\RRT^3_2$ (\ref{thm:RRT32-presv-D3}) & \\
      \hline
      $\Delta^0_2$ & $\COH$ (\ref{thm:COH-presv-D2}) & $\SADS$ (\ref{thm:SADS-n-presv}) \\
	& $\EM$ (\ref{thm:EM-presv-D2}) & \\
      \hline
      $\Sigma^0_1, \Pi^0_1$ & $\RT^2_2, \FS, \ART^n_{<\infty, d}$ (\ref{cor:presv-P1}) & $\RT^n_2$ ($n > 2$, \cite{Jockusch:1972.Ramsey}) \\
      \hline
    \end{tabu}
    \bigskip
    \caption{Preservations and Non-preservations}\label{table}
\end{table}

From Table \ref{table} and Proposition \ref{pro:sep-by-presv}, we can derive some consequences about provability strength. We present two examples here. The first is by examining the preservations and non-preservations of $\Delta^0_2$-definitions.

\begin{theorem}\label{thm:sep-1}
Let $\Phi$ be the conjunction of $\COH$, $\WKL$, $\RRT^2_2$, $\Pi^0_1 \operatorname{G}$ and $\EM$. Over $\RCA$, $\Phi$ does not imply any of $\SADS$, $\STS^2$, $\TS^2$ and $\FS^2$.
\end{theorem}

Theorem \ref{thm:sep-1} strengthens the following known results: $\RCA + \COH + \WKL \not\vdash \SADS$ (Hirschfeldt and Shore \cite{Hirschfeldt.Shore:2007}), $\RCA + \Pi^0_1 \operatorname{G} \not\vdash \SADS$ (Hirschfeldt, Shore and Slaman \cite{Hirschfeldt.Shore.ea:2009}), $\RCA + \RRT^2_2 \not\vdash \SADS$ (Csima and Mileti \cite{Csima.Mileti:2009.rainbow}), $\RCA + \RRT^2_2 \not\vdash \TS^2$ (Kang \cite{Kang:2013}), $\RCA + \EM \not\vdash \SADS$ (Lerman, Solomon and Towsner \cite{Lerman.Solomon.ea:2013}), and $\RCA + \EM \not\vdash \STS^2$ (Patey, unpublished). But the approach here is more uniform. Moreover, Proposition \ref{pro:sep-by-presv} allows us to stack $\Pi^1_2$ propositions with weak definability strength together. This is another advantage of our approach.

The next example follows from the $\Delta^0_3$ row of Table \ref{table}.

\begin{theorem}\label{thm:sep-RRT32-TS32}
Over $\RCA$, $\RRT^3_2$ does not imply any of $\STS^3, \TS^3, \FS^3$.
\end{theorem}

Yet there are questions around Table \ref{table}. From Proposition \ref{pro:sep-by-presv} and the $\Delta^0_2$ row of the table, we can derive an almost complete classification of well-known $\Pi^1_2$ propositions below $\ACA$. However, we know just a little at rows above $\Delta^0_2$. Moreover, there is an interesting phenomenon: for either $\COH$ or $\EM$, we prove that each admits preservation of $\Delta^0_2$-definitions and preservation of definitions beyond the $\Delta^0_2$ level (implied by that of $\RT^2_2$). So it is natural to conjecture that these two preservations can be combined for both $\COH$ and $\EM$.

\begin{conjecture}
Both $\COH$ and $\EM$ admit preservation of the arithmetic hierarchy.
\end{conjecture}

\bibliographystyle{plain}

\end{document}